\documentclass[11pt]{article}

\usepackage[english]{babel}
\usepackage{amsmath,amsfonts,amssymb,amsthm}
\usepackage{latexsym}
\usepackage{makeidx}
\usepackage{amscd}

\usepackage{vmargin}
\setmarginsrb{11.6mm}{10mm}{11.6mm}{20mm}%
            {8mm}{8mm}{8mm}{8mm}

\numberwithin{equation}{section}

\newtheorem{deff}{Definition}[section]
\newtheorem{lemma}[deff]{Lemma}
\newtheorem{theorem}[deff]{Theorem}
\newtheorem{corollary}[deff]{Corollary}

\newtheorem{proposition}[deff]{Proposition}
\newtheorem{fact}[deff]{Fact}
\newtheorem{em-example}[deff]{Example}
\newtheorem{em-def}[deff]{Definition}        
\newtheorem{em-remark}[deff]{Remark}         
\newtheorem{em-question}[deff]{Question}

\newtheorem{claim}[deff]{Claim}

\newenvironment{example}{\begin{em-example} \em }{ \end{em-example}}
\newenvironment{remark}{\begin{em-remark} \em }{\end{em-remark}}

\newcommand{\R}{\mathbb R}
\newcommand{\N}{\mathbb N}

\newcommand{\Q}{\mathbb Q}
\newcommand{\Z}{\mathbb Z}

\newcommand{\T}{\mathbb{T}}

\newcommand{\f}{\phi}

\def\ent{\mathrm{ent}}
\def\End{\mathrm{End}}
\def\Aut{\mathrm{Aut}}
\def\pgp{\mathrm{Pol}}

\def\P{\mathbf{P}}
\def\QQ{\mathfrak Q}

\global\def\hull#1{\langle{#1}\rangle}

\input xy
\xyoption{all}

\title{The Pinsker subgroup of an algebraic flow}

\author{Dikran Dikranjan
\\{\footnotesize {\tt  dikran.dikranjan@dimi.uniud.it}} 
\\{\footnotesize Dipartimento di Matematica e Informatica,}
\\{\footnotesize Universit\`{a} di Udine,}
\\{\footnotesize Via delle Scienze, 206 - 33100 Udine, Italy} 
 \and Anna Giordano Bruno
\\{\footnotesize {\tt  anna.giordanobruno@math.unipd.it}} 
\\{\footnotesize Dipartimento di Matematica Pura e Applicata,}
\\{\footnotesize Universit\`a di Padova,}
\\{\footnotesize Via Trieste, 63 - 35121 Padova}
 }

\date{}

\begin{document}

\maketitle


\abstract{The algebraic entropy $h$, defined for endomorphisms $\phi$ of abelian groups $G$, measures the growth of the trajectories of non-empty finite subsets $F$ of $G$ with respect to $\phi$. We show that this growth can be either polynomial or exponential. The greatest $\phi$-invariant subgroup of $G$ where this growth is 
polynomial coincides with the greatest $\phi$-invariant subgroup $\P(G,\phi)$ of $G$ (named Pinsker subgroup of $\phi$) such that $h(\phi\restriction_{\P(G,\phi)})=0$. 
We obtain also an alternative characterization of $\P(G,\phi)$ from the point of view of the quasi-periodic points of $\phi$. 

This gives the following application in ergodic theory: for every  continuous injective endomorphism  $\psi$ of a compact abelian group $K$ there exists a largest  $\psi$-invariant closed subgroup $N$ of $K$ such that $\psi\restriction_{N}$ is ergodic; furthermore, the induced endomorphism $\overline{\psi}$ of the quotient $K/N$ has zero topological entropy.}

\section{Introduction}

The algebraic entropy $\ent$ of endomorphisms of abelian groups was first defined by Adler, Konheim and McAndrew in \cite{AKM}, then studied by Weiss \cite{W} and more recently in \cite{DGSZ}.  This function is studied also in other recent papers  \cite{AADGH,AGB,SZ1,SZ2,Z}. 

\medskip
The definition given by Weiss is appropriate for torsion abelian groups, since any endomorphism $\phi$ of an abelian group $G$ has  $\ent(\phi)=\ent(\phi\restriction_{t(G)})$ (so endomorphisms of torsion-free abelian groups have zero algebraic entropy). Peters \cite{Pet} modified the definition of algebraic entropy for automorphisms of arbitrary abelian groups, using in the definition the non-empty finite subsets instead of the finite subgroups used by Weiss. This definition was extended in \cite{DG} to endomorphisms of abelian groups, as follows. Let $G$ be an abelian group and $\phi\in\End(G)$. For a non-empty subset $F$ of $G$ and for any positive integer $n$, the \emph{$n$-th $\phi$-trajectory} of $F$ is 
$$
T_n(\phi,F)=F+\phi(F)+\ldots+\phi^{n-1}(F),
$$ 
and the \emph{$\phi$-trajectory} of $F$ is $T(\phi,F)=\sum_{n\in\N}\phi^n(F).$
For $F$ finite, let 
$$
\tau_{\phi,F}(n)=|T_n(\phi,F)|;
$$ 
when there is no possibility of confusion we write only $\tau_F(n)$, omitting the endomorphism $\phi$. For the function $\tau_F:\N_+\to \N$ the limit
\begin{equation}\label{lim-eq}
H(\phi,F)=\lim_{n\to\infty}\frac{\log\tau_{\phi,F}(n)}{n}
\end{equation}
exists (see \cite{DG}), and $H(\phi,F)$ is the \emph{algebraic entropy} of $\phi$ with respect to $F$. The \emph{algebraic entropy} of $\phi$ is 
$$
h(\phi)=\sup_{F\in[G]^{<\omega}}H(\phi,F).
$$
When $\phi$ is an automorphism, the algebraic entropy defined in this way has the same values as that defined (in a different way) by Peters.
Clearly, $h$ coincides with $\ent$ for endomorphisms of torsion abelian groups, but while the function $\ent$ trivializes for torsion-free abelian groups, $h$ remains non-trivial also in this case.

\smallskip
Many properties satisfied by $\ent$ are shared also by $h$ (see Fact \ref{conjugation_by_iso}). The main ones are the Addition Theorem  \ref{AT} 
and the Uniqueness Theorem \ref{UT} proved in \cite{DG}. The main tool used for the proof of those theorems is the ``Algebraic Yuzvinski Formula'' -- here Theorem \ref{Yuz} -- which will be used also in this paper (only in the proof of Proposition \ref{h=0->qp}).

\medskip
For a measure preserving  transformation $\phi$ of a measure space $(X, {\mathcal B}, \mu)$ the {\em Pinsker $\sigma$-algebra} $\mathfrak P(\phi)$ of $\phi$ is the greatest $\sigma$-subalgebra of ${\mathcal B}$ such that $\phi$ restricted to $(X, \mathfrak P(\phi),\mu\restriction_{{\mathcal B}})$ has zero measure entropy. Note that $id_X: (X, {\mathfrak B},\mu) \to (X, {\mathfrak P}(\phi),\mu\restriction_{{\mathcal B}})$ is measure preserving, so $(X, {\mathfrak P}(\phi),\mu\restriction_\mathcal B)$ is a factor of $(X, {\mathfrak B},\mu)$ (see \cite{Wa}).

\smallskip
The situation is similar in topological dynamics. Let $(X,\phi)$ be a \emph{topological flow}, i.e. a homeomorphism $\phi:X\to X$ of a compact Hausdorff space  $X$. A \emph{factor} $(\pi,(Y,\psi))$ of $(X,\phi)$ is a topological flow $(Y,\psi)$ together with a continuous surjective map $\pi:X\to Y$ such that $\pi\circ \phi=\psi\circ\pi$. In \cite{BL} (see also \cite{KerLi}) it is proved that a topological flow $(X,\phi)$ admits a largest factor with zero topological entropy, called \emph{topological Pinsker factor} (we recall the definition of topological entropy in Section \ref{top-sec}).  As mentioned in \cite{KerLi}, this can be deduced from the fact that the class of all homeomorphisms $\f:X\to X$ with zero topological entropy is closed under taking products and subspaces. 

\smallskip
In analogy with the topological case, we call an \emph{algebraic flow} a pair $(G,\phi)$, where $G$ is an abelian group and $\phi\in\End(G)$.
For an algebraic flow $(G,\phi)$, we say that a subgroup $H$ of $G$ is \emph{$\phi$-invariant} (or just \emph{invariant} when $\phi$ is clear from the context) if $\phi(H)\subseteq H$. A \emph{factor} of $(G,\phi)$ is an algebraic flow of the form $(G/H,\overline\phi)$, where $H$ is a $\phi$-invariant subgroup of $G$ and $\overline\phi$ is the endomorphism induced by $\phi$ on the quotient $G/H$.

 An algebraic flow may fail to have a largest factor with zero algebraic entropy $h$ even when the underlying group is torsion; indeed, zero algebraic entropy is \emph{not} preserved under taking products, as Example \ref{no:prod:stab} shows.
So the algebraic counterpart of the Pinsker factor in the case of an algebraic flow cannot be a {\em factor}. This 
motivated us to introduce the notion of Pinsker subgroup as follows.

\begin{deff}
Let $(G,\phi)$ be an algebraic flow. The \emph{Pinsker subgroup} of $G$ with respect to $\phi$ is the greatest $\f$-invariant subgroup $\P(G,\f)$ of $G$ such that $h(\f\restriction_{\P(G,\f)})=0$.
\end{deff}

As shown by Proposition \ref{P:existence} the Pinsker subgroup exists and has a number of nice properties. The aim of this paper is to study the properties of this subgroup and characterize it two distinct ways. 

The first one involves the very definition of entropy measuring the growth of the function $\tau_{\phi,F}(n)$.  Since
\begin{equation}\label{2}
|F|\leq\tau_{\phi,F}(n)\leq|F|^n\mbox{ for every }n\in\N_+,
\end{equation}
 the growth of $\tau_{\phi,F}$ is always at most exponential, so $H(\phi,F)\leq \log |F|$ in \eqref{lim-eq} (note (\ref{2}) implies also that $\tau_{\phi,F}(n)=1$ for every $n\in\N_+$ precisely when $F$ is a singleton). This justifies the following definition.

\begin{deff}
Let $(G,\phi)$ be an algebraic flow and let $F$ be a non-empty finite subset of $G$. Then we say that:
\begin{itemize}
\item[(a)] $\phi$ has \emph{exponential growth} with respect to $F$  (denoted by  $\phi\in\mathrm{Exp}_F$) if there exists $b\in\R$, $b>1$, such that $\tau_{\phi,F}(n)\geq b^n$ for every $n\in\N_+$.
\item[(b)] $\phi$ has \emph{polynomial growth} with respect to $F$ (denoted by $\f\in\pgp_F$)  if there exists $P_F(x)\in\Z[x]$ such that $\tau_{\f,F}(n)\leq P_F(n)$ for every $n\in\N_+$;
\item[(c)] $\phi$ has \emph{polynomial growth} (denoted by $\f\in\pgp$) if $\f\in\pgp_F$ for every non-empty finite subset $F$ of $G$.
\end{itemize}
\end{deff}

It is not hard to see that for a fixed non-empty finite subset $F$ of an abelian group $G$ an endomorphism $\phi:G\to G$ has exponential growth with respect to $F$ if and only if $H(\phi,F)>0$ (see Theorem \ref{exp}). On the other hand, if $H(\phi,F)=0$, then $\tau_{\phi,F}$ has growth less than exponential, yet it is not obvious that  $\f\in\pgp_F$ holds. To clarify this issue, in Section \ref{PGP-sec} we show that for any algebraic flow $(G,\phi)$ there exists a greatest $\f$-invariant subgroup  $\pgp(G,\f)$ of $G$ where the restriction of $\f$ has polynomial growth. Since $\phi\in\pgp_F$ immediately yields $H(\phi,F)=0$ (see Corollary \ref{Q<PGP<P}), this entails $\pgp(G,\f)\subseteq \P(G,\f)$. To prove that these two subgroups coincide, one needs the following remarkable

\bigskip 
\noindent {\bf Dichotomy Theorem.} {\em Every algebraic flow $(G,\phi)$ has either exponential or polynomial growth with respect to any fixed non-empty finite subset $F$ of $G$. }
\bigskip 
 
 In other words, $H(\phi,F)=0$ if and only if $\phi\in\pgp_F$ holds for any fixed non-empty  finite subset $F$ of $G$. 
By taking the universal quantifiers with respect to $F$, one can easily deduce from the Dichotomy Theorem that $h(\phi) = 0$ if and only if $\phi\in\pgp$ (see Corollary \ref{0<->pgp}), i.e., $ \P(G,\f) = \pgp(G,\f)$. This characterizes the Pinsker subgroup as the greatest $\phi$-invariant subgroup of $G$ on which the restriction of $\phi$ has polynomial growth.
The proof of this theorem requires a significant effort (see Theorem \ref{exp}). Indeed, we deduce the Dichotomy Theorem  from the equality $ \P(G,\f) = \pgp(G,\f)$, 
which is a corollary of our Main Theorem involving also  
 another important  subgroup related to the algebraic flow  $(G,\phi)$.

\smallskip
The second way to characterize  the Pinsker subgroup involves quasi-periodic points as follows. For an algebraic flow $(G,\f)$ 
 the \emph{$\phi$-torsion subgroup} of $G$ 
 $$
 t_\phi(G)=\{x\in G: |T(\phi, \hull{x})|<\infty\}
 $$
was introduced in \cite{DGSZ}. Note that $t_\phi(G)\subseteq t(G)$, and $t_\phi(G)$ consists of the quasi-periodic points of $\phi$ in $t(G)$.
The equality $ \P(G,\f) = t_\phi(G)$ for a torsion $G$ can be attributed to \cite{DGSZ} (see also Corollary \ref{t=P} for a proof). In other words, 
in this case the Pinsker subgroup coincides with the subgroup of all quasi-periodic points. 
%

In Section \ref{QP-sec} we elaborate this key idea in the general case, by introducing the counterpart of $
 t_\phi(G)$ in the non-torsion case. Namely, the smallest $\f$-invariant subgroup $\mathfrak Q(G,\f)$ of $G$ such that the induced endomorphism $\overline\phi$ of   
$G/\mathfrak Q(G,\f)$ has no non-zero quasi-periodic points (so in particular, it contains all quasi-periodic points of $\phi$ in $G$). 
This subgroup  captures important features of the dynamics of the endomorphism $\f$  from this internal  point of view, making no recurse to the algebraic entropy of $\f$. 
In Section \ref{PGP-sec} we show that $\mathfrak Q(G,\f)\subseteq  \pgp(G,\f)$ (see Corollary \ref{Q<pgp}). Along with the above mentioned inclusion $\pgp(G,\f)\subseteq \P(G,\f)$ this 
gives a chain $\mathfrak Q(G,\f)\subseteq  \pgp(G,\f) \subseteq  \P(G,\f)$. It turns out that  all three subgroups coincide:

\bigskip 
\noindent {\bf Main Theorem.} 
\emph{For every algebraic flow $(G,\phi)$, then $\mathfrak Q(G,\f)=\pgp(G,\f)=\P(G,\f)$.}
\bigskip 

This alternative description of the Pinsker subgroup is proved in Theorem \ref{Q=PGP=P}.
 
\medskip
The interest in the algebraic entropy $h$ stems from the study of topological entropy of continuous endomorphisms of compact abelian groups (see Theorem \ref{Peters}), 
but the growth of the function $\tau_{\f,F}(n)$ has deep roots also in the non-abelian group theory (\cite{Wolf}, where the growth of $\tau_{id_G,F}$ may already 
be exponential for some non-abelian group $G$, see Remark \ref{WolfA}). 

As an application of our results we show that for every continuous injective endomorphism  $\psi$ of a compact abelian group $K$ there exists a largest  $\psi$-invariant closed subgroup $N$ of $K$ such that $\psi\restriction_{N}$ is ergodic; furthermore, the induced endomorphism of the quotient $K/N$ has 
 zero topological entropy.


\subsubsection*{Notation and terminology}

We denote by $\mathbb Z$, $\mathbb N$, $\mathbb N_+$, $\Q$ and $\R$ respectively the set of integers, the set of natural numbers, the set of positive integers, the set of rationals and the set of reals. For $m\in\mathbb N_+$, we use $\mathbb Z(m)$ for the finite cyclic group of order $m$. 

Let $G$ be an abelian group. With a slight divergence with the standard use, we denote by $[G]^{<\omega}$ the set of all non-empty finite subsets of $G$. If $H$ is a subgroup of $G$, we indicate this by $H\leq G$. The subgroup of torsion elements of $G$ is $t(G)$, while $D(G)$ denotes the divisible hull of $G$. For a cardinal $\alpha$ we denote by $G^{(\alpha)}$ the direct sum of $\alpha$ many copies of $G$, that is $\bigoplus_\alpha G$. 

Moreover, $\End(G)$ is the ring of all endomorphisms of $G$. We denote by $0_G$ and $id_G$ respectively the endomorphism of $G$ which is identically $0$ and the identity endomorphism of $G$.  For $F\in[G]^{<\omega}$ and $n\in\N_+$ let 
$$
F_{(n)}=\underbrace{F+\ldots+F}_n.
$$
The hyperkernel of $\phi\in\End(G)$ is $\ker_\infty\phi=\bigcup_{n\in\N}\ker\f^n$. The endomorphism $\overline\phi:G/\ker_\infty\phi\to G/\ker_\infty\phi$ induced by $\phi$ is injective.


\section{Background on algebraic entropy}

We start this section giving the so-called Algebraic Yuzvinski Formula.

\begin{theorem}[Algebraic Yuzvinski Formula]\label{Yuz}
For $n\in \N_+$ an automorphism $\phi$ of $\Q^n$  is described by a matrix $A \in GL_n(\Q)$.  Then
\begin{equation}\label{yuz-eq}
h(\f)=\log s+ \sum_{|\lambda_i|>1}\log |\lambda_i|,
\end{equation}
where $\lambda_i$ are the eigenvalues of $A$ and $s$ is the least common multiple of the denominators of the coefficients of the (monic) characteristic polynomial of $A$.
\end{theorem}

The proof of this theorem is discussed in Section \ref{top-sec}. 

We recall here some other definition.
For an algebraic flow $(G,\phi)$ and $F\in[G]^{<\omega}$,
since $T(\phi,F)$ need not be a subgroup of $G$, we denote by $V(\phi,F)=\hull{T(\phi,F)}=\hull{\f^n(F):n\in\N}$ the smallest $\phi$-invariant subgroup of $G$ containing $F$ (and consequently containing also $T(\phi,F)$). If $F=\{g\}$ for some $g\in G$ we write simply $V(\phi,g)$. Note that $V(\phi,F)=\sum_{g\in F}V(\f,g)$.

\smallskip
In the next claim we collect some technical properties of the trajectories that will be used in the paper.

\begin{claim}\label{Tf-Tf^k}
Let $(G,\phi)$ be an algebraic flow, and let $n,k\in\N_+$. Then:
\begin{itemize}
\item[(a)]$T_n(\f^k,F)\subseteq T_{nk-n+1}(\f,F)$ $($consequently, $\tau_{\f^k,F}(n)\leq\tau_{\f,F}(nk-n+1)$$)$;
\item[(b)]$T_{nk}(\f,F)=T_n(\f^k,T_k(\f,F))$ $($consequently, $\tau_{\f,F}(nk)=\tau_{\f^k,T_k(\f,F)}(n)$$)$.
\item[(c)]In particular, $T(\f,F)=T(\f^k,T_k(\f,F))$.
\end{itemize}
\end{claim}

 In the next claim we see that the identical endomorphism of any abelian group has polynomial growth and we deduce in Example \ref{h(id)=0} 
that the algebraic entropy of the identical endomorphism is zero. 

\smallskip

\begin{claim}\label{claim0}
Let $G$ be an abelian group and $F\in[G]^{<\omega}$. Then $|F_{(n)}|\leq (n+1)^{|F|}$ for every $n\in\N_+$.
\end{claim}
\begin{proof}
Let $F=\{f_1,\ldots,f_t\}$ and let $n\in\N_+$. If $x\in F_{(n)}$, then $x=\sum_{i=1}^t m_i f_i$, for some $m_i\in\N$ with $\sum_{i=1}^t m_i=n$. Then $0\leq m_i\leq n$ for every $i=1,\ldots,t$, that is, $(m_1,\ldots,m_t)\in\{0,1,\ldots,n\}^t$, and so $|F_{(n)}|\leq (n+1)^t$.
\end{proof}

\begin{remark}\label{WolfA}
It should be noted that if the group $G$ is non-abelian, then the function $n\mapsto |F_{(n)}|$ may have exponential growth \cite{Wolf}. 
\end{remark}

\begin{example}\label{h(id)=0}
For any abelian group $G$, the identical endomorphism $id_G$ has $h(id_G)=0.$
Indeed, let $F\in[G]^{<\omega}$. By Claim \ref{claim0}, $\tau_{id_G,F}(n)=|T_n(id_G,F)|=|F_{(n)}|\leq (n+1)^{|F|}$ for every $n\in\N_+$.
Hence 
$$
H(id_G,F)=\lim_{n\to\infty}\frac{\log\tau_{id_G,F}(n)}{n}\leq\lim_{n\to\infty}\frac{|F|\log(n+1)}{n}=0
.$$ 
Since $F$ was chosen arbitrarily, we can conclude that $h(id_G)=0$.
\end{example}

In the next fact we collect the basic properties of the algebraic entropy. The properties in (b), (c) and (d) were proved in \cite{Pet} in the case of automorphisms and then generalized to endomorphisms in \cite{DG}.

\begin{fact}\label{conjugation_by_iso}\emph{\cite{DG}}
Let $(G,\phi)$ be an algebraic flow.
\begin{itemize}
\item[(a)] If $H$ is a $\phi$-invariant subgroup of $G$ and $\overline\phi:G/H\to G/H$ is the endomorphism induced by $\phi$, then $h(\phi)\geq\max\{h(\phi\restriction_H),h(\overline\phi)\}$.
\item[(b)] Let $(H,\eta)$ be another algebraic flow. If $\f$ and $\eta$ are conjugated (i.e., there exists an isomorphism $\xi:G\to H$ such that $\phi=\xi^{-1}\eta\xi$), then $h(\phi)=h(\eta)$.
\item[(c)] For every $k\in\N$, then $h(\phi^k) = k\cdot h(\phi)$. If $\phi$ is an
automorphism, then $h(\phi^k) = |k|\cdot h(\phi)$ for every $k\in \Z$.
\item[(d)] If $G=G_1\times G_2$ and $\phi_i\in\End(G_i)$, for $i=1,2$, then $h(\phi_1\times\phi_2)=h(\phi_1)+h(\phi_2)$.
\item[(e)] If $G$ is a direct limit of $\phi$-invariant subgroups $\{G_i:i\in I\}$, then $h(\phi)=\sup_{i\in I}h(\phi\restriction_{G_i})$.
\end{itemize}
\end{fact}

\begin{example}\label{beta}
For any abelian group $K$ the \emph{right Bernoulli shift} is $\beta_{K}:K^{(\N)}\to K^{(\N)}$ defined by 
$$
(x_0,x_1,x_2,\ldots)\mapsto(0,x_0,x_1,\ldots).
$$ 
It is proved in \cite{DGSZ} that $h(\beta_{\Z(p)})=\ent(\beta_{\Z(p)})=\log p$, where $p$ is a prime. Then it follows that $h(\beta_\Z)=\infty$ in view of Fact \ref{conjugation_by_iso}(a) (as noted also in \cite{DG}).
Then $h(\beta_K)=\log|K|$ (see \cite{DG,DGSZ}), with the usual convention that $\log|K|=\infty$ if $|K|$ is infinite.
\end{example}

The following are the main properties of $h$, namely, the so-called Addition Theorem and Uniqueness Theorem proved in \cite{DG}. They are inspired by their counterparts for $\ent$, proved in \cite{DGSZ}, which were inspired by the similar properties of the topological entropy.

\begin{theorem}[Addition Theorem]\label{AT}
Let $(G,\phi)$ be an algebraic flow, $H$ a $\phi$-invariant subgroup of $G$ and $\overline\phi:G/H\to G/H$ the endomorphism induced by $\phi$. Then $$h(\f)=h(\f\restriction_H) + h(\overline\f).$$
\end{theorem}

\begin{theorem}[Uniqueness Theorem]\label{UT}
The algebraic entropy $h$ of the endomorphisms of the abelian groups is characterized as the unique collection $h=\{h_G:G\ \text{abelian group}\}$ of functions $h_G:\End(G)\to\R_+$ such that:
\begin{itemize}
\item[(a)] $h_G$ is invariant under conjugation for every abelian group $G$;
\item[(b)] if $\phi\in\End(G)$ and the group $G$ is a direct limit of $\phi$-invariant subgroups $\{G_i:i\in I\}$, then $h_G(\phi)=\sup_{i\in I}h_{G_i}(\phi\restriction_{G_i})$;
\item[(c)] the Addition Theorem holds for $h$;
\item[(d)] $h_{K^{(\N)}}(\beta_K)=\log|K|$ for any finite abelian group $K$;
\item[(e)] the Algebraic Yuzvinski Formula holds for $h_\Q$ restricted to the automorphisms of $\Q$.
\end{itemize}
\end{theorem}

The following two useful properties of the algebraic entropy are needed in the proof of our main results, namely, Proposition \ref{h=0->qp} and Theorem \ref{exp}.

\begin{lemma}\label{AA}
Let $(G,\phi)$ be an algebraic flow. If $G$ is torsion-free and $\widetilde\f:D(G)\to D(G)$ denotes the unique extension of $\f$ to the divisible hull $D(G)$ of $G$, then $h(\widetilde\f)=h(\f)$.
\end{lemma}
\begin{proof}
It is obvious that $h(\widetilde\f)\geq h(\f)$  by Fact \ref{conjugation_by_iso}(a). 

Let $F\in[D(G)]^{<\omega}$. Then there exists $m\in\N_+$ such that $mF\subseteq G$. Let $\mu_m(x) = mx$ for every $x\in D(G)$. Then $\mu_m$ is an automorphism of $D(G)$ that commutes with $\widetilde \f$. Moreover, $T_n(\f, mF)=T_n(\f,\mu_m(F))=\mu_m(T_n(\widetilde \f,F))$. In particular, $\tau_{\f, mF}(n) = \tau_{\widetilde \f,F}(n)$. Hence, $H(\f,mF)=H(\widetilde \f,F)$. Since $H(\f,mF) \leq h(\f)$, we conclude that $H(\widetilde \f,F)\leq h(\f)$. By the arbitrariness of $F$, this gives $h(\widetilde\f)\leq h(\f)$.  
\end{proof}

\begin{lemma}\label{proexp}
Let $(G,\phi)$ be an algebraic flow and $F\in[G]^{<\omega}$. If $H(\phi,F)=0$, then $h(\f\restriction_{V(\phi,F)})=0$.
\end{lemma}
\begin{proof}
From $H(\f,F)=0$, it follows that $H(\f,F_{(m)})=0$ for every $m\in\N_+$; indeed, fixed $m\in\N_+$, for every $n\in\N_+$, we have $T_n(\f,F_{(m)})=T_n(\f,F)_{(m)}$, and so $\tau_{\f,F_{(m)}}(n)\leq\tau_{\f,F}(n)^m$. 

Note that
\begin{equation}\label{*}
V(\phi,F)=\bigcup_{m\in\N_+}T_m(\phi,F_{(m)}).
\end{equation}
Moreover, for $m\in\N_+$,
\begin{equation}\label{**}
H(\phi,F_{(m^2)})=0\ \ \text{implies}\ \ H(\phi,T_m(\phi,F_{(m)}))=0,
\end{equation}
since $T_n(\phi,T_m(\phi,F_{(m)}))\subseteq T_{n+m}(\phi,F_{(m^2)})$ for every $n\in\N_+$.
Now, for every $F'\in[V(\phi,F)]^{<\omega}$, by \eqref{*} there exists $m\in\N_+$ such that $F'\subseteq T_m(\phi,F_{(m)})$. By \eqref{**} $H(\phi,T_m(\phi,F_{(m)}))=0$ and so also $H(\phi,F')=0$. By the arbitrariness of $F'$, this proves that $h(\phi\restriction_{V(\phi,F)})=0$.
\end{proof}


\section{The Pinsker subgroup}

The next proposition proves the existence of the Pinsker subgroup for any algebraic flow.

\begin{proposition}\label{P:existence}
Let $(G,\phi)$ be an algebraic flow. The Pinsker subgroup $\P(G,\f)$ of $G$ exists.
\end{proposition}
\begin{proof}
Let $\mathcal F=\{H\leq G: H\ \text{$\phi$-invariant},\ h(\phi\restriction_H)=0\}$.

We start proving that
\begin{equation}\label{sum_in_F}
\text{if $H_1,\ldots,H_n\in\mathcal F$, then $H_1+\ldots+H_n\in\mathcal F$.}
\end{equation}
Let $H_1,H_2\in\mathcal F$. Consider $\xi=\phi\restriction_{H_1}\times\phi\restriction_{H_2}:H_1\times H_2\to H_1\times H_2$. By Fact \ref{conjugation_by_iso}(d) $h(\xi)=0$. Since $H_1+H_2$ is a quotient of $H_1\times H_2$, the quotient endomorphism $\overline{\xi}:H_1\to H_2$ induced by $\xi$ has $h(\overline{\xi})=0$ by Fact \ref{conjugation_by_iso}(a).
Since $\phi\restriction_{H_1+H_2}$ is conjugated to $\overline{\xi}$, it follows that $h(\phi\restriction_{H_1+H_2})=0$ by Fact \ref{conjugation_by_iso}(b). Proceeding by induction it is clear how to prove \eqref{sum_in_F}.

\smallskip
Let $P=\hull{H:H\in\mathcal F}$, and let $F\in[P]^{<\omega}$.
Then there exist $n\in\N_+$ and $H_1,\ldots,H_n\in\mathcal F$ such that $F\subseteq H_1+\ldots+H_n$. By \eqref{sum_in_F} $h(\f\restriction_{H_1+\ldots+H_n})=0$, so in particular $H(\f\restriction_P,F)=H(\f\restriction_{H_1+\ldots+H_n},F)=0$. Since $F$ is arbitrary, this proves that $h(\f\restriction_P)=0$. By the definition of $P$, if $H$ is a $\f$-invariant subgroup of $G$ with $h(\f\restriction_H)=0$, then $H\subseteq P$. Hence $P=\P(G,\f)$.
\end{proof}

It is clear that for an algebraic flow $(G,\phi)$, $h(\phi)=0$ if and only if $G=\P(G,\phi)$. In the opposite direction we consider the following property. 

\begin{deff}
Let $(G,\phi)$ be an algebraic flow. We say the $\phi$ has \emph{completely positive algebraic entropy} if $h(\phi\restriction_H)>0$ for every $\phi$-invariant subgroup $H$ of $G$. We denote this by $h(\phi)>\!\!>0$.
\end{deff}

Clearly, $h(\phi)>\!\!> 0$ if and only if $\P(G,\phi)=0$.

\smallskip
The definition is motivated by its topological counterpart: 
a topological flow $(X,\psi)$ has \emph{completely positive topological entropy} if all its non-trivial factors have positive topological entropy  \cite{BL}.

\begin{lemma}\label{pi<p}
Let $(G,\phi)$ be an algebraic flow and let $H$ be a $\phi$-invariant subgroup of $G$. Then:
\begin{itemize}
\item[(a)]$\P(H,\phi\restriction_H)=\P(G,\phi)\cap H$;
\item[(b)]for $\pi:G\to G/H$ the canonical projection and $\overline\f:G/H\to G/H$ the endomorphism induced by $\phi$, $\pi(\P(G,\f))\subseteq \P(G/H,\overline\f)$.
\end{itemize}
\end{lemma}
\begin{proof}
(a) Since $h(\phi\restriction_{\P(H,\phi\restriction_H)})=0$, it follows that $\P(H,\phi\restriction_H)\subseteq \P(G,\phi)\cap H$. Since $\P(G,\phi)\cap H\subseteq \P(G,\phi)$, we have $h(\phi\restriction_{\P(G,\phi)\cap H})=0$ and so $\P(G,\phi)\cap H\subseteq\P(H,\phi\restriction_H)$.

\smallskip
(b) Follows immediately from Fact \ref{conjugation_by_iso}(a).
\end{proof}

The inclusion in item (b) of the above lemma cannot be replaced by equality (e.g., if $G= \Q$, $H = \Z$ and $\phi$ is defined by $\phi(x) = 2x$ for $x\in \Q$, then 
$\P(G,\phi)=0$, while $\P(G/H,\overline\f)=G/H$). 

Lemma \ref{pi<p} shows in particular the stability properties of the class of endomorphisms with zero algebraic entropy under taking invariant subgroups and quotients over invariant subgroups. By Fact \ref{properties}(d) this class is preserved also under taking finite direct products, and so also under taking arbitrary direct sums. Indeed, if $G=\bigoplus_{i\in I}G_i$, then $G$ is direct limit of $\bigoplus_{i\in F}G_i$ where $F$ runs over all non-empty finite subsets of $I$, and so Fact \ref{properties}(e) applies.
Example \ref{no:prod:stab} below shows that the class of endomorphisms with zero algebraic entropy is not preserved under taking arbitrary infinite direct products.

\begin{lemma}\label{Pi-tor}\emph{\cite{DGSZ}}
Let $(G,\phi)$ be an algebraic flow and assume that $G$ is torsion. Then:
\begin{itemize}
\item[(a)] $h(\phi\restriction_{t_\phi(G)})=0$ and $t_{\overline{\phi}}(G/t_\f(G))=0$ for the endomorphism $\overline{\phi}:G/t_\f(G)\to G/t_\f(G)$ induced by $\phi$;
\item[(b)] if $H$ is a $\phi$-invariant subgroup of $G$ and $h(\phi\restriction_{H})=0$, then $H\subseteq t_\phi(G)$.
\end{itemize}
\end{lemma}

\begin{example}\label{no:prod:stab}
For a prime $p$, let $G=\prod_{n\in\N_+}\Z(p)^n$ and for every $n\in\N_+$ consider the right shift $\beta_n:\Z(p)^n\to\Z(p)^n$ defined by $(x_1,x_2,\ldots,x_n)\mapsto(0,x_1,\ldots,x_{n-1})$. Since $\Z(p)^n$ is finite, $h(\beta_n)=\ent(\beta_n)=0$. On the other hand, let $\phi=\prod_{n\in\N_+} \beta_n:G\to G$; then $h(\phi)=\ent(\phi)>0$. Indeed, for $n\in\N_+$, let $x_n\in \Z(p)^n$ be such that $\beta_{\Z(p)}^{n-1}(x_n)\ne 0$, hence $x=(x_n)_{n\in\N_+}\in G$ has infinite trajectory under $\phi$. By Lemma \ref{Pi-tor}(b) we have $h(\phi)=\ent(\phi)>0$.
\end{example}

From items (a) and (b) of Lemma \ref{Pi-tor} we also obtain immediately

\begin{corollary}\label{t=P}
Let $(G,\phi)$ be an algebraic flow. If $G$ is torsion, then $\P(G,\f)=t_\phi(G)$.
\end{corollary}

Let $(G,\phi)$ be an algebraic flow. A point $x\in G$ is a \emph{periodic} point of $\phi$ if there exists $n\in\N_+$ such that $\phi^n(x) =x$. Let $P_1(G,\f)$ be the subset of $G$ of all periodic points of $\phi$. Obviously, $P_1(G,\f)$ is a $\phi$-invariant subgroup of $G$. Moreover, $x\in G$ is a \emph{quasi-periodic} point of $\phi$ if there exist $n>m$ in $\N$ such that $\phi^n(x)=\phi^m(x)$. Let $Q_1(G,\phi)$ be the subset of $G$ of all quasi-periodic points of $\phi$. Hence $Q_1(G,\f)$ is a $\phi$-invariant subgroup of $G$ as well, since $Q_1(G,\f)= \bigcup_{n\in\N} \f^{-n}(P_1(G,\f))$.

Clearly, the quasi-periodic points of an injective endomorphism $\phi$ are periodic, that is, $Q_1(G,\f)=P_1(G,\f)$. 

If $P_1(G,\f)=G$, we say that $\f$ is \emph{locally periodic}, i.e., for every $x\in G$ there exists $n\in\N_+$ such that $\phi^n(x)=x$. Moreover, $\f$ is \emph{periodic} if there exists $n\in\N_+$ such that $\phi^n(x)=x$ for every $x\in G$. Analogously, if $Q_1(G,\f)=G$, we say that $\f$ is \emph{locally quasi-periodic}, i.e., for every $x\in G$ there exist $n>m$ in $\N$ such that $\phi^n(x)=\phi^m(x)$. Finally, $\phi$ is \emph{quasi-periodic} if there exist $n>m$ in $\N$ such that $\phi^n(x)=\phi^m(x)$ for every $x\in G$.

\begin{proposition}\label{t=Q1}
Let $(G,\phi)$ be an algebraic flow. Then 
$$
t_\phi(G)= t(G)\cap Q_1(G,\f)\subseteq Q_1(G,\f)\subseteq\P(G,\f).
$$ 
If $G$ is torsion, then all these three subgroups coincide. 
\end{proposition}
\begin{proof}
The inclusion $t_\phi(G)\subseteq t(G)$ is obvious. If $x\in t_\phi(G)$, then $V(\phi,x)$ is a finite $\f$-invariant subgroup of $G$ and so $V(\phi,x)\subseteq Q_1(G,\f)$, hence $t_\phi(G)\subseteq Q_1(G,\phi)$. Since every quasi-periodic point has finite trajectory, for $x\in t(G)\cap Q_1(G,\f)$ one has $x\in t_\phi(G)$. This proves the first equality. Moreover, $Q_1(G,\f)\subseteq\P(G,\f)$, since for every $F\in[G]^{<\omega}$ with $F\subseteq Q_1(G,\phi)$, there exists $m\in\N_+$ such that $T_n(\phi,F)=T_m(\phi,F)$ for every $n\in\N$, $n\geq m$, hence $H(\phi,F)=0$. This proves that $h(\phi\restriction_{Q_1(G,\phi)})=0$, and so $Q_1(G,\phi)\subseteq \P(G,\phi)$.

If $G$ is torsion, $\P(G,\f)=t_\f(G)$ by Corollary \ref{t=P}.
\end{proof}

 Note that always $
t_{id_G}(G)= t(G)\subseteq Q_1(G,id_G)\subseteq\P(G,id_G).$

 The example of torsion abelian groups gives the motivation and the idea on how to approach the Pinsker subgroup of arbitrary abelian groups making use of quasi-periodic points. To this end we need a generalization of the notions of periodic and quasi-periodic points.


\section{Generalized quasi-periodic points} \label{QP-sec}

Let $(G,\phi)$ be an algebraic flow. We extend the definition of $P_1(G,\phi)$ setting by induction:
\begin{itemize}
\item[(a)] $P_0(G,\f)=0$, and for every $n\in\N$
\item[(b)] $P_{n+1}(G,\f)=\{x\in G: (\exists n\in \N_+)\ \phi^n(x)-x\in P_n(G,\phi)\}$.
\end{itemize}
This gives an increasing chain 
$$
P_0(G,\f)\subseteq P_1(G,\f)\subseteq\ldots\subseteq P_n(G,\f)\subseteq\ldots.
$$ 
We show below that all members of this chain are $\phi$-invariant subgroups of $G$. Our interest in these subgroups is motivated by the fact that they are contained in $\P(G,\f)$. 
Although this fact can be proved directly, we shall obtain it as a consequence of a stronger property (see Remark \ref{uff-rem}).

In order to approximate better $\P(G,\f)$ we need to enlarge these subsets introducing for every $n\in\N$ appropriate counterparts of $Q_1(G,\phi)$, as follows. Define
\begin{itemize}
\item[(a)] $Q_0(G,\f)=0$, and for every $n\in\N$
\item[(b)] $Q_{n+1}(G,\f)=\{x\in G: (\exists n>m \ \text{in}\ \N)\ (\phi^n-\phi^m)(x)\in Q_n(G,\phi)\}$.
\end{itemize}

We get an increasing chain 
$$
Q_0(G,\f)\subseteq Q_1(G,\f)\subseteq \ldots\subseteq Q_n(G,\f)\subseteq \ldots.
$$
Again we show below that the members of this chain are $\phi$-invariant subgroups of $G$.
One can prove by induction that for injective endomorphisms $\phi$ of an abelian group $G$, $Q_n(G,\f)=P_n(G,\f)$ for every $n\in\N$.

\begin{deff}
For an algebraic flow $(G,\phi)$, let 
$\QQ(G,\f)=\bigcup_{n\in\N} Q_n(G,\f)$.
\end{deff}

\begin{proposition}\label{PnQn}
Let $(G,\phi)$ be an algebraic flow.
\begin{itemize}
\item[(a)] For every $n\in\N$, $P_n(G,\phi)$ is a $\phi$-invariant subgroup of $G$, and for the induced endomorphism $\overline \f_n:G/ P_n(G,\f) \to G/ P_n(G,\f)$ and the canonical projection $\pi_n: G \to G/P_n(G,\f)$, 
\begin{equation}\label{P-eq}
P_{n+1}(G,\f)=\pi_n^{-1}(P_1(G/P_n(G,\f),\overline \f_n))\ (\text{i.e.,}\ P_{n+1}(G,\f)/P_n(G,\f)=P_1(G/P_n(G,\f),\overline\f_n)).
\end{equation}
\item[(b)] For every $n\in\N$, $Q_n(G,\phi)$ is a $\phi$-invariant subgroup of $G$, and for the induced endomorphism $\overline \f_n:G/  Q_n(G,\f) \to G/  Q_n(G,\f)$ and the canonical projection $\pi_n: G \to G/ Q_n(G,\f)$, 
\begin{equation}\label{Q-eq}
Q_{n+1}(G,\f)=\pi_n^{-1}(Q_1(G/ Q_n(G,\f),\overline \f_n))\ (\text{i.e.,}\  Q_{n+1}(G,\f)/ Q_n(G,\f)= Q_1(G/ Q_n(G,\f),\overline\f_n)).
\end{equation}
\item[(c)] $\QQ(G,\f)$ is a $\f$-invariant subgroup of $G$.
\end{itemize}
\end{proposition}
\begin{proof} 
(a) We proceed by induction. The cases $n=0$ and $n=1$ are trivial.  
Let $n\in\N$ and assume that $P_n(G,\phi)$ is a $\phi$-invariant subgroup of $G$.

Let $x\in P_{n+1}(G,\f)$. This is equivalent to the existence of $k\in\N_+$ such that $\phi^k(x)- x\in P_n(G,\f)$. This occurs if and only if $\pi_n(\phi^k(x)-x)=0$ in $G/ P_n(G,\f)$. Since $\overline\f_n^k(\pi_n(x))-\pi_n(x)=\pi_n(\phi^k(x)-x)$, this is equivalent to $\pi_n(x)\in P_1(G/P_n(G,\f),\overline\f_n)$, that is, $x\in\pi_n^{-1}(P_1(G/P_n(G,\f),\overline \f_n))$. So this proves \eqref{P-eq}. 

\smallskip 
Since $P_1(G/P_n(G,\f),\overline \f_n)$ is a $\overline\phi_n$-invariant subgroup of $G/P_n(G,\f)$, its counterimage $P_{n+1}(G,\f)$ is a $\phi$-invariant subgroup of $G$.

\smallskip
(b) Argue as in (a). 

\smallskip
(c) By (a), (b) and its definition, $\QQ(G,\f)$ is a $\f$-invariant subgroup of $G$.
\end{proof}

It is worth noting the following clear property.

\begin{lemma}\label{Q=0<->Q1=0}
If $(G,\phi)$ is an algebraic flow, then $\QQ(G,\phi)=0$ if and only if $Q_1(G,\phi)=0$. 
\end{lemma}

Recall that a subgroup $H$ of an abelian group $G$ is \emph{pure} in $G$ if and only if $G/H$ is torsion-free.

\begin{lemma}
Let $(G,\phi)$ be an algebraic flow. If $G$ is torsion-free, then $Q_n(G,\f)$, for every $n\in\N$, and $\QQ(G,\f)$ are pure in $G$.
\end{lemma}
\begin{proof}
We proceed by induction. The case $n=0$ is trivial.

Let $k\in\N$ and $x\in  Q_1(G,\f)\cap k G$. Then $x=k y$ for some $y\in G$ and there exist $n>m$ in $\N$ such that $\phi^n(x)=\phi^m(x)$. Consequently, $k\phi^n(y)=k \phi^m(y)$, which yields $\phi^n(y)=\phi^m(y)$ as $G$ is torsion-free. So $y\in Q_1(G,\f)$, i.e., $x\in k Q_1(G,\f)$. This proves that $Q_1(G,\phi)$ is pure in $G$.

In the sequel we denote $G/Q_m(G,\f)$ by $G_m$ for $m\in\N$. Assume that $Q_{n}(G,\f)$ is pure in $G$ (so $G_n$ is torsion-free) for some $n\in\N$. 
We have to show that $G_{n+1}$ is torsion-free as well. Let $\overline\phi_n:G_n\to G_n$ be the endomorphism induced by $\phi$.
Then $G_{n+1}=G/Q_{n+1}(G,\f)\cong G_n/(Q_{n+1}(G,\f)/Q_n(G,\f))=G_n/Q_1(G_n,\overline\phi_n)$, where the last equality holds in view of \eqref{Q-eq}. 
By the case $n=1$, applied to the torsion-free group $G_n$, the latter quotient is torsion-free, so $G_{n+1}$ is torsion-free as well.

As an increasing union of pure subgroups, $\QQ(G,\f)$ is pure itself.
\end{proof}

\begin{claim}\label{claim}
Let $(G,\phi)$ be an algebraic flow, $H$ a $\f$-invariant subgroup of $G$ and $x\in G$. If $\phi^k(x)-x\in H$ for some $k\in\N_+$, then $\phi^{nk}(x)-x\in H$ for every $n\in\N_+$.
\end{claim}
\begin{proof}
Let $n\in\N_+$, let $\overline\phi:G/H\to G/H$ be the endomorphism induced by $\phi$, and $\pi:G\to G/H$ the canonical projection. Note that $\phi^{m}(x)-x\in H$ for some $m\in \N_+$ if and only if $\overline\phi^{m}(\pi(x))=\pi(x)$ and we proceed by induction on $n$.  
\end{proof}

Note that $\ker_\infty\phi\subseteq Q_1(G,\phi)$, so also $\ker_\infty\phi\subseteq Q_n(G,\phi)$ for every $n\in\N_+$, and in particular $\ker_\infty\phi\subseteq\QQ(G,\phi)$.

\begin{lemma}\label{Q=P+K}
Let $(G,\phi)$ be an algebraic flow. Then $Q_n(G,\f)=P_n(G,\f)\oplus\ker_\infty\f$ for every $n\in\N_+$.
\end{lemma}
\begin{proof}
We start proving that $P_n(G,\phi)\cap \ker_\infty\phi=0$, proceeding by induction. For $n=1$, let $x\in P_1(G,\f)\cap \ker_\infty\f$. Then there exist 
$s,t\in\N_+$ such that $\f^s(x)=x$ and $\f^t(x)=0$. Consequently, $x=\f^{st}(x)=0$. 
Assume now that $n\in\N_+$ and that $P_{n}(G,\f)\cap\ker_\infty\f=0$. Let $x\in P_{n+1}(G,\f)\cap\ker_\infty\f$. Then there exist $s,t\in\N_+$ such that $\f^s(x)-x\in P_{n}(G,\f)$ and $\f^t(x)=0$. By Claim \ref{claim}, $\f^{st}(x)-x\in P_{n}(G,\f)$. Since $\phi^{st}(x)=0$, this yields   $x\in P_{n}(G,\f)\cap \ker_\infty\f$. By the inductive hypothesis $x=0$.

\smallskip
For the quotient $G/\ker_\infty\f$ and the canonical projection $\pi:G\to G/\ker_\infty\f$, the induced 
 endomorphism $\overline\f:G/\ker_\infty\f\to G/\ker_\infty\f$ is injective. Hence
 \begin{equation}\label{inj-eq}
Q_n(G/\ker_\infty\f,\overline\f)=P_n(G/\ker_\infty\f,\overline\f)
\end{equation}
for every $n\in\N$.
Then, in order to show that $Q_n(G,\f)=P_n(G,\f)+\ker_\infty\f$ for every $n\in\N_+$, it suffices to show that for every $n\in\N$:
\begin{itemize}
\item[(a)] $\pi(P_n(G,\f))=P_n(G/\ker_\infty\f,\overline\f)$, and
\item[(b)] $\pi(Q_n(G,\f))=Q_n(G/\ker_\infty\f,\overline\f)$.
\end{itemize}
(a) Clearly $\pi(P_n(G,\f))\subseteq P_n(G/\ker_\infty\f,\overline\f)$ for every $n\in\N$.
So we prove by induction the converse inclusion. The case $n=0$ is trivial.
Assume now that $n\in\N$ and that $\pi(P_{n}(G,\f))=P_{n}(G/\ker_\infty\f,\overline\f)$.
Let $x\in G$ be such that $\pi(x)\in P_{n+1}(G/\ker_\infty\f,\overline\f)$. Then there exists $s\in\N_+$ such that $\overline\f^s(\pi(x))-\pi(x)\in P_{n}(G/\ker_\infty\f,\overline\f)$. Since $\pi(\f^s(x)-x)=\overline\f^s(\pi(x))-\pi(x)\in P_{n}(G/\ker_\infty\f,\overline\f)$, we conclude that
\begin{equation}\label{fs}
\f^s(x)-x\in P_{n}(G,\f)+\ker_\infty\f.
\end{equation}
Hence there exists $k\in\N_+$ such that $\f^k(\f^s(x)-x)\in P_{n}(G,\f)$ and so $\f^s(\f^k(x))-\f^k(x)\in P_{n}(G,\f)$, that is, $\f^k(x)\in P_{n+1}(G,\f)$.  By Claim \ref{claim} applied to \eqref{fs}, and since $P_{n+1}(G,\f)$ is $\phi$-invariant,
$$\f^{sk}(x)-x\in P_{n}(G,\f)+\ker_\infty\f\ \text{and} \ \f^{sk}(x)\in P_{n+1}(G,\f).$$
Hence $x\in\ker_\infty\phi+P_{n+1}(G,\phi)$, and consequently $\pi(x)\in\pi(P_{n+1}(G,\f))$.

\smallskip
(b) Let $n\in\N$. Clearly $\pi(P_n(G,\f))\subseteq\pi(Q_n(G,\f))\subseteq Q_n(G/\ker_\infty\f,\overline\f)$. Moreover,  $Q_n(G/\ker_\infty\f,\overline\f)=P_n(G/\ker_\infty\f,\overline\f)=\pi(P_n(G,\f))$, by \eqref{inj-eq} and (a). Hence all these subgroups coincide, and in particular 
$\pi(Q_n(G,\f))=Q_n(G/\ker_\infty\f,\overline\f)$.
\end{proof}

It is easy to see that if $G$ is an abelian group and $\f\in\Aut(G)$, then 
 $\f(Q_n(G,\f))= Q_n(G,\f)$ and $Q_n(G,\f^{-1})=Q_n(G,\f)$ for every $n\in\N$. Consequently, 
 $\f(\QQ(G,\f))=\QQ(G,\f)$ and $\mathfrak Q(G,\f^{-1})=\QQ(G,\f)$. 
 
\begin{proposition}\label{NewProp}
Let $(G,\phi)$ be an algebraic flow.  Then
$Q_n(G,\f)=\phi^{-1}(Q_n(G,\f))$ (i.e., the induced endomorphism $\overline\f_n:G/ Q_n(G,\f)\to G/ Q_n(G,\f)$ is injective) for every $n\in\N$. Hence, $\QQ(G,\f)=\phi^{-1}(\QQ(G,\f))$ (i.e., the induced endomorphism  $G/\QQ(G,\f) \to G/ \QQ(G,\f)$ 
 is injective). 
\end{proposition}
\begin{proof} 
The equality $Q_1(G,\f)=\phi^{-1}(Q_1(G,\f))$ easily follows from the definitions. 
Then an inductive argument using (\ref{Q-eq}) applies. 
\end{proof}

From Proposition \ref{NewProp} one easily obtains: 

\begin{corollary}\label{NewCoro} 
If $(G,\phi)$ is an algebraic flow with surjective $\phi$, then the induced endomorphisms $G/\QQ(G,\f) \to G/ \QQ(G,\f)$ and $G/ Q_n(G,\f)\to G/ Q_n(G,\f)$ ($n\in\N$) are automorphisms. 
\end{corollary}

\begin{example}\label{Z^N}
\begin{itemize}
\item[(a)] Let $G=\Z^2$ and let $\f$ be the automorphism of $G$ defined by $\f(x,y)=(x,x+y)$, that is, by the matrix $\begin{pmatrix}1 & 1 \\ 0 & 1\end{pmatrix}$. The subgroup $H=\Z\times\{0\}$ is $\f$-invariant.
Moreover, $Q_1(G,\f)=H$, while $\QQ(G,\f)= Q_2(G,\f)=G$. Consequently, 
$$0=Q_0(G,\phi)\subset Q_1(G, \f) \subset Q_2(G, \f)=\QQ(G,\phi)=G.$$ 
\end{itemize}
With similar examples one can show that the Loewy length of $\QQ(G,\f)$ may be arbitrarily large (up to $\omega$):
\begin{itemize}
\item[(b)] Let $G=\Z^{(\N)}=\bigoplus_{n=1}^\infty\hull{e_n}$. Let $\phi$ be the automorphism of $G$ given by the matrix 
$$\begin{pmatrix} 
1 & 1 & 1 & \ldots \\
0 & 1 & 1 & \ldots \\
0 & 0 & 1 & \ldots \\
\vdots & \vdots & \ddots & \ddots
\end{pmatrix}.$$
For every $n\in\N_+$ let $G_n=\hull{e_1,\ldots,e_n}$. Then $G_n=P_n(G,\phi)=Q_n(G,\phi)$ and $\QQ(G,\f)=G$. In particular, we have the following strictly increasing chain
$$0=Q_0(G,\phi)\subset Q_1(G,\phi)\subset Q_2(G,\phi)\subset\ldots\subset Q_n(G,\phi)\subset\ldots\subset\QQ(G,\phi)=G.$$
\end{itemize}
\end{example}

The next lemma is clear.

\begin{lemma}\label{Q_H}
Let $(G,\phi)$ be an algebraic flow and $H$ a $\phi$-invariant subgroup of $G$. Then:
\begin{itemize}
\item[(a)] $Q_n(H,\phi\restriction_H)=Q_n(G,\phi)\cap H$ for every $n\in\N$; so
\item[(b)]$\QQ(H,\phi\restriction_H)=\QQ(G,\phi)\cap H$.
\end{itemize}
\end{lemma}


The following notion will be motivated and used in Section \ref{top-sec}.

\begin{deff}\label{alg-erg}
Let $(G,\phi)$ be an algebraic flow. Call $\f$ \emph{algebraically ergodic} if $\mathfrak Q(G,\phi)=0$ (that is, $\phi$ has no non-trivial quasi-periodic point). 
\end{deff}

Observe that for an abelian group $G$, the endomorphism $0_G$ is algebraically ergodic if and only if $G=0$.

\smallskip
The next proposition shows that, for $(G,\phi)$ an algebraic flow, $\QQ(G,\phi)$ is the smallest $\phi$-invariant subgroup $H$ of $G$ such that the induced endomorphism $\overline\phi:G/H\to G/H$ is algebraically ergodic.

\begin{proposition}\label{P.T}
Let $(G,\phi)$ be an algebraic flow. Then:
\begin{itemize}
\item[(a)] the endomorphism $\overline{\f}: G/\QQ(G,\f)\to G/\QQ(G,\f)$ induced by $\phi$ is algebraically ergodic, i.e., 
$$
\QQ(G/\QQ(G,\f),\overline\f)=0;
$$ 
\item[(b)] if for some $\f$-invariant subgroup $H$ of $G$ the endomorphism $\overline{\f}: G/H\to G/H$ induced by $\phi$ is algebraically ergodic, then 
$H\supseteq \QQ(G,\f)$.
\end{itemize}
\end{proposition}
\begin{proof}
(a) According to Proposition \ref{NewProp} it suffices to check that $\overline\phi$ has no non-zero periodic points. 
Let $\pi:G\to G/\QQ(G,\f)$ be the canonical projection, and assume that $\pi(x)\in G/\QQ(G,\f)$ is a periodic point of $\overline{\phi}$ for some $x\in G$. Then there exists $n\in \N_+$ such that $\phi^n(x)-x\in \QQ(G,\f)$. Consequently $\phi^n(x)-x\in  Q_s(G,\phi)$ for some $s\in\N$. This yields $x\in  Q_{s+1}(G,\phi)\subseteq\QQ(G,\f)$. Thus $\pi(x)=0$ in $G/\QQ(G,\f)$. Hence $Q_1(G/\QQ(G,\f),\overline\f)=0$. By Lemma \ref{Q=0<->Q1=0}, $\QQ(G/\QQ(G,\f),\overline\f)=0$.

\smallskip
(b) It suffices to see that $Q_n(G,\phi) \subseteq H$ for each $n\in\N$. We shall prove it by induction on $n\in\N$, the case $n=0$ being trivial. If $n\in\N$ and $Q_{n}(G,\phi) \subseteq H$, then for $x\in  Q_{n+1}(G,\phi)$ and $\pi:G\to G/H$ the canonical projection, $\pi(x)\in G/H$ is a quasi-periodic point of the induced endomorphism $\overline{\phi}: G/H\to G/H$, as $Q_{n+1}(G,\phi)/Q_n(G,\phi)=Q_1(G/Q_n(G,\phi),\overline\phi_n)$ by \eqref{Q-eq}, where $\overline\phi_n:G/Q_n(G,\phi)\to G/Q_n(G,\phi)$ is the induced endomorphism. So our hypothesis yields $\pi(x)=0$, that is, $x\in H$. 
\end{proof}


\section{The polynomial growth}\label{PGP-sec}

\begin{example}\label{lex}
Let $G$ be an abelian group.
\begin{itemize}
\item[(a)] 
Then $id_G\in\pgp$ by Claim \ref{claim0} and it is obvious that $0_G\in\pgp$.
\item[(b)] Moreover, $\phi\restriction_{\ker_\infty\phi}\in\pgp$.
Indeed, for every $F\in[\ker_\infty\phi]^{<\omega}$, there exists $m\in\N_+$ such that $\phi^m(F)=0$. Then $\tau_F(n)=\tau_F(m)$ for every $n\in\N$ with $n\geq m$. In particular, $\phi\in\pgp_F$.
\end{itemize}
\end{example}

The following obvious claim underlines the fact   that the property of having polynomial growth is in some sense ``local''.

\begin{claim}\label{pgp-local}
Let $(G,\phi)$ be an algebraic flow and $F\in[G]^{<\omega}$. The following conditions are equivalent:
\begin{itemize}
\item[(a)] $\f\in\pgp_F$;
\item[(b)] $\f\restriction_{V(\phi,F)}\in\pgp_F$;
\item[(c)] $\f\restriction_H\in\pgp_F$ for every $\phi$-invariant subgroup $H$ of $G$ such that $F\subseteq H$.
\end{itemize}
\end{claim}

Using the argument from the proof of Lemma \ref{proexp}, one can prove that 
the above equivalent conditions imply the stronger one $\f\restriction_{V(\phi,F)}\in\pgp$. We are not giving this proof, since this stronger property can be deduced from Lemma \ref{proexp} and Theorem \ref{exp}. 

The next proposition gives basic properties of endomorphisms with polynomial growth. These are analogous to the properties considered for the algebraic entropy.

\begin{proposition}\label{properties}
Let $(G,\phi)$ be an algebraic flow.
\begin{itemize}
\item[(a)] Let $H$ be a $\phi$-invariant subgroup of $G$ and $\overline\phi:G/H\to G/H$ the endomorphism induced by $\phi$. If $\phi\in\pgp$, then $\phi\restriction_H\in\pgp$ and $\overline\phi\in\pgp$. 
\item[(b)] Let $(H,\eta)$ be another algebraic flow. If $\phi$ and $\eta$ are conjugated, (i.e., there exists an isomorphism $\xi:G\to H$ such that $\phi=\xi^{-1}\eta\xi$), then $\f\in\pgp$ if and only if $\eta\in\pgp$. More precisely, if $F'\in[H]^{<\omega}$, then $\eta\in\pgp_{F'}$ if and only if $\f\in\pgp_{\xi^{-1}(F')}$ (if $F\in[G]^{<\omega}$, then $\f\in\pgp_F$ if and only if $\eta\in\pgp_{\xi(F)}$).
\item[(c)] If $G=G_1\times G_2$ and $\phi_i\in\End(G_i)$, for $i=1,2$, then $\phi_1\times\phi_2\in\pgp$ if and only if $\phi_1\in\pgp$ and $\phi_2\in\pgp$.
\item[(d)] Let $G$ be a direct limit of $\f$-invariant subgroups $\{G_i:i\in I\}$. If $\f\restriction_{G_i}\in\pgp$ for every $i\in I$, then $\f\in\pgp$.
\item[(e)] If $G$ is torsion-free, then $\f\in\pgp$ if and only if $\widetilde\f\in\pgp$, where $D(G)$ is the divisible hull of $G$ and $\widetilde\f:D(G)\to D(G)$ is the unique extension of $\phi$ to $D(G)$.
\end{itemize}
\end{proposition}
\begin{proof}
(a) Let $F\in[H]^{<\omega}$. Then there exists $P_F(x)\in\Z[x]$ such that $|T_n(\f\restriction_H,F)|=|T_n(\f,F)|\leq P_F(n)$ for every $n\in\N_+$. Let now $\pi:G\to G/H$ be the canonical projection and let $F'\in[G/H]^{<\omega}$. Then there exists $F\in[G]^{<\omega}$ such that $|F|=|F'|$ and $\pi(F)=F'$. Then $|T_n(\overline\f,F')|=|\pi(T_n(\f,F))|\leq |T_n(\f,F)|\leq P_F(n)$ for every $n\in\N_+$.

\smallskip
(b) If $F'\in[H]^{<\omega}$ and $n\in\N_+$, then $T_n(\xi\f\xi^{-1},F')=\xi(T_n(\f,\xi^{-1}(F')))$. Equivalently, if $F\in[G]^{<\omega}$ and $n\in\N_+$, then $T_n(\f,F)=\xi^{-1}(T_n(\xi\f\xi^{-1},\xi(F))$.

\smallskip
(c) If $\phi_1\times\phi_2\in\pgp$, then $\phi_1\in\pgp$ and $\phi_2\in\pgp$ by (a). So assume that $\phi_1\in\pgp$ and $\phi_2\in\pgp$. Let $F\in [G_1\times G_2]^{<\omega}$. Then $F$ is contained in some $F_1\times F_2$, where $F_i\in [G_i]^{<\omega}$ for $i=1,2$. By definition, for $i=1,2$, there exists $P_{F_i}(x)\in\Z[x]$ such that $|T_n(\f_i,F_i)|\leq P_{F_i}(n)$ for every $n\in\N_+$. 
Hence $|T_n(\phi_1\times\phi_2,F)|\leq|T_n(\phi_1\times\phi_2,F_1\times F_2)|\leq|T_n(\phi_1,F_1)|\cdot|T_n(\phi_2,F_2)|\leq P_{F_1}(n)P_{F_2}(n)$ for every $n\in\N_+$, and $P_{F_1}(x)P_{F_2}(x)\in\Z[x]$.

\smallskip
(d) Let $F\in[G]^{<\omega}$. Since $G=\bigcup_{i\in I}G_i$ and the family $\{G_i:i\in I\}$ is directed, $F\subseteq G_j$ for some $j\in I$. By hypothesis $\f\restriction_{G_j}\in\pgp_F$. By Claim \ref{pgp-local} $\f\in\pgp_F$. By the arbitrariness of $F$ this proves that $\f\in\pgp$.

\smallskip
(e) Let $F\in[D(G)]^{<\omega}$. Then there exists $m\in\N_+$ such that $mF\subseteq G$. Let $\mu_m(x) = mx$ for every $x\in D(G)$. Then $\mu_m$ is an automorphism of $D(G)$ that commutes with $\widetilde \f$. Moreover, $T_n(\f, mF)=T_n(\f,\mu_m(F))=\mu_m(T_n(\widetilde \f,F))$. In particular, $\tau_{\widetilde\phi,F}=\tau_{\f, mF}$. Hence $\widetilde\f\in\pgp_F$, as $\phi\in\pgp_{mF}$.
\end{proof}

The following property is related to powers of endomorphisms.

\begin{lemma}\label{pgp-lemma}
Let $(G,\phi)$ be an algebraic flow, let $k\in\N_+$ and $F\in[G]^{<\omega}$.
\begin{itemize}
\item[(a)] If $\f\in\pgp_F$, then $\f^k\in\pgp_F$.
\item[(b)] If $\f^k\in\pgp_{T_k(\f,F)}$, then $\f\in\pgp_F$.
\end{itemize}
In particular, $\f\in\pgp$ if and only if $\f^k\in\pgp$.
\end{lemma}
\begin{proof}
Let $n\in\N_+$. We can suppose without loss of generality that $0\in F$ (so that the cardinality of the trajectories grows with $n$). 

\smallskip
(a) Since $\f\in\pgp_F$, there exists $P_F(x)\in\Z[x]$ such that $\tau_{\phi,F}(n)\leq P_F(n)$ for every $n\in\N_+$. By Claim \ref{Tf-Tf^k}(a), for every $n\in\N_+$, $$\tau_{\phi^k,F}(n)\leq\tau_{\f,F}(kn-n+1)\leq P_F((k-1)n+1).$$ This shows that $\f^k\in\pgp_F$.

\smallskip
(b) Since $\f^k\in\pgp_{T_k(\f,F)}$, there exists a polynomial $P(x)\in\Z[x]$,  depending only on $F$ and the fixed $k$, 
such that $\tau_{\phi^k,T_k(\phi,F)}(n)\leq P(n)$ for every $n\in\N_+$. By Claim \ref{Tf-Tf^k}(b), for every $n\in\N_+$, 
$$
\tau_{\phi,F}(n)\leq\tau_{\f,F}(nk)=\tau_{\f^k,T_k(\f,F)}(n)\leq P(n).
$$ 
 This proves that $\f\in\pgp_F$.

\smallskip
The last assertion follows directly from (a) and (b).
\end{proof}

\begin{example}
Let $(G,\phi)$ be an algebraic flow.
If either $\phi^k=id_G$ (i.e., $\phi$ is periodic) or $\phi^k=0_G$ (i.e., $\f$ is nilpotent) for some $k\in\N_+$, then $\phi\in\pgp$.
This follows immediately from Example \ref{lex} and Lemma \ref{pgp-lemma}.
\end{example}

\begin{deff}
For an algebraic flow $(G,\phi)$, let $\pgp(G,\f)$ be the greatest $\phi$-invariant subgroup of $G$ such that $\f\restriction_{\pgp(G,\f)}\in\pgp$.
\end{deff}

The proof of the following lemma is similar to the proof of the existence of the Pinsker subgroup of an algebraic flow given in Proposition \ref{P:existence}.

\begin{lemma}\label{pgp-existence}
Let $(G,\phi)$ be an algebraic flow. Then $\pgp(G,\f)$ exists.
\end{lemma}
\begin{proof}
Let $\mathcal F=\{H\leq G: H\ \text{$\phi$-invariant},\ \phi\restriction_H\in\pgp\}$.

We start proving that
\begin{equation}\label{sum_in_F-}
\text{if $H_1,\ldots,H_n\in\mathcal F$, then $H_1+\ldots+H_n\in\mathcal F$.}
\end{equation}
Let $H_1,H_2\in\mathcal F$. Consider $\xi=\phi\restriction_{H_1}\times\phi\restriction_{H_2}:H_1\times H_2\to H_1\times H_2$. By Proposition \ref{properties}(c) $\xi\in\pgp$. Since $H_1+H_2$ is a quotient of $H_1\times H_2$, consider the quotient endomorphism $\overline{\xi}:H_1\to H_2$ induced by $\xi$. By Proposition \ref{properties}(a) $\overline{\xi}\in\pgp$.
Since $\phi\restriction_{H_1+H_2}$ is conjugated to $\overline{\xi}$, it follows that $\phi\restriction_{H_1+H_2}\in\pgp$ by Proposition \ref{properties}(b). Proceeding by induction it is easy to prove \eqref{sum_in_F-}.

\smallskip
Let $P=\hull{H:H\in\mathcal F}$; it suffices to prove that $\phi\restriction_P\in\pgp$. To this end let $F\in[P]^{<\omega}$.
Then there exist $n\in\N_+$ and $H_1,\ldots,H_n\in\mathcal F$ such that $F\subseteq H_1+\ldots+H_n$. By \eqref{sum_in_F-} $\f\restriction_{H_1+\ldots+H_n}\in\pgp$, so in particular $\f\restriction_P\in\pgp_F$ by Lemma \ref{pgp-local}. Hence $\f\restriction_P\in\pgp$. By definition, if $H$ is a $\f$-invariant subgroup of $G$ with $\f\restriction_H\in\pgp$, then $H\subseteq P$. Hence $P=\pgp(G,\f)$.
\end{proof}

For an algebraic flow $(G,\phi)$, one may ask whether $\pgp(G/\pgp(G,\phi),\overline\phi)=0$, where $\overline\phi:G/\pgp(G,\phi)\to G/\pgp(G,\phi)$ is the endomorphism induced by $\phi$. This property in fact holds true and will follow from results below (see Proposition \ref{P.T} and Theorem \ref{Q=PGP=P}).

\begin{claim}\label{fingen,locper->per}
Let $(G,\phi)$ be an algebraic flow. If $G=V(\f,F)$ for some $F\in[G]^{<\omega}$ and $\f$ is locally periodic, then $\f$ is periodic.
\end{claim}
\begin{proof}
There exists $m\in\N_+$ such that $\f^m\restriction_F=id_F$. Hence $\f^m=id_G$.
\end{proof}



The next claim generalizes Claim \ref{claim0}. 

\begin{claim}\label{claim1}
Let $(G,\phi)$ be an algebraic flow and $F\in[G]^{<\omega}$. 
If $(\phi- id_G)^m = 0$ for some $m\in\N_+$, then $\tau_{\f,F}(n)\leq (n^m+1)^{m|F|}$ for every $n\in\N_+$.
\end{claim}
\begin{proof}
Let $|F|=t$, and let $s=\phi- id_G$, so that $s^m=0$. Then
\begin{itemize}
\item[(1)]$\phi^n(x)=x+ C_n^1 s(x)+ C_n^2 s^2(x)+\ldots+C_n^{m-1}s^{m-1}(x)$ for every $x\in G$ and $n\in\N$ with $n\geq m$; therefore
\item[(2)]$\phi^n(F)\subseteq F+ s(F)_{(C_n^1)}+s^2(F)_{(C_n^2)}+\ldots+ s^{m-1}(F)_{(C_n^{m-1})}$ for every $F\in[G]^{<\omega}$ with $0\in F$ and every $n\in\N$ with $n\geq m$.
\end{itemize}
Hence
$$\phi^n(F)\subseteq F+ s(F)_{(n)}+s^2(F)_{(n^2)}+\ldots+ s^{m-1}(F)_{(n^{m-1})}\ \text{for every $F\in[G]^{<\omega}$ with $0\in F$ and every $n\in\N$ with $n\geq m$}.$$
Therefore
\begin{itemize}
\item[(3)] $T_n(\phi,F)\subseteq F_{(n)}+ s(F)_{(n^2)}+s^2(F)_{(n^3)}+\ldots+s^{m-1}(F)_{(n^m)}$ for every $F\in[G]^{<\omega}$ with $0\in F$ and every $n\in\N$ with $n\geq m$.
\end{itemize}
If $|F|=t$, then $|s^k(F)|\leq t$ for every $k\in\N$, so applying Claim \ref{claim0} to each $s^k(F)$ and (3), we find
\begin{align*}
\tau_{\f,F}(n) &\leq|F_{(n)}|\cdot|s(F)_{(n^2)}|\cdot|s^2(F)_{(n^3)}|\cdot\ldots\cdot|s^{m-1}(F)_{(n^m)}| \\
& \leq (n+1)^t\cdot(n^2+1)^t\cdot(n^3+1)^t\cdot\ldots\cdot(n^m+1)^t\leq(n^m+1)^{mt}.
\end{align*}
\end{proof}

\begin{proposition}\label{QQ-pgp}
Let $(G,\phi)$ be an algebraic flow. Then $\f\restriction_{\QQ(G,\f)}\in\pgp$. 
\end{proposition}
\begin{proof}
We start proving that
\begin{equation}\label{Pm-1->Pm-gen}
\phi\restriction_{P_m(G,\phi)}\in\pgp \ \text{for every}\ m\in\N.
\end{equation}
Fix $m\in\N$ and $F\in[P_m(G,\phi)]^{<\omega}$. We want to prove that $\f\in\pgp_F$, that is, we have to find $P_F(x)\in\Z[x]$ such that $\tau_{\f,F}(n)\leq P_F(n)$ for every $n\in\N_+$.
Since $V(\f,F)$ is a $\f$-invariant subgroup of $P_m(G,\phi)$ and since $P_m(V(\f,F),\f\restriction_{V(\f,F)})=P_m(G,\f)\cap V(\f,F)=V(\f,F)$, by Claim \ref{pgp-local} we can suppose without loss of generality that $G=V(\f,F)$.

For every $k\in\{0,\ldots,m-2\}$ let $$G_k = P_{m-k}(G,\f)/P_{m-k-1}(G,\f).$$ Let $\overline\f_k:G/P_{m-k-1}(G,\f)\to G/P_{m-k-1}(G,\f)$ be the endomorphism induced by $\phi$. Every $\overline\f_k\restriction_{G_k}: G_k \to G_k$ is locally periodic, because $G_k=P_1(G/P_{m-k-1}(G,\f),\overline\f_k)$. 
Since $G=V(\phi,F)$, $G/P_{m-k-1}(G,\f)=V(\overline\f_k,\pi_k(F))$, where $\pi_k:G\to G/P_{m-k-1}(G,\f)$ is the canonical projection. Since $G/P_{m-k-1}(G,\f)$ is a finitely generated $\Z[X]$-module, and $\Z[X]$ is noetherian, every $G_k$ is finitely generated as a $\Z[X]$-module, that is, there exists $F_k\in[G_k]^{<\omega}$ such that $G_k=V(\overline\f_k,F_k)$.
By Claim \ref{fingen,locper->per} each $\overline\f_k$ is periodic on $G_k$. Then there exists $w_k\in\N_+$ such that $\overline\f^{w_k}_k\restriction_{G_k}=id_{G_k}$. Let $w=w_0\cdot\ldots\cdot w_{m-2}$. Then $\overline\f^w_k\restriction_{G_k}=id_{G_k}$ for every $k\in\{0,\ldots,m-2\}$.

By Claim \ref{Tf-Tf^k}(c), $G=V(\f,F)=\hull{T(\f,F)}=\hull{T(\f^w,T_w(\f,F)}$; moreover, $P_l(G,\f)=P_l(G,\f^w)$ for every $l\in\N_+$.
By Lemma \ref{pgp-lemma}(b), $\f^w\in\pgp_{T_w(\f,F)}$ implies $\f\in\pgp_F$. So without loss of generality we can replace $\f^w$ by $\f$, that is, we can suppose that $\overline\f_k\restriction_{G_k}=id_{G_k}$ for every $k$.

Let $s=\f-id_G$; then $s(P_{m-k}(G,\f)) \subseteq P_{m-k-1}(G,\f)$ for every $k\in\{0,\ldots,m-2\}$. In particular, $s^m = 0$. By Claim \ref{claim1}(b) $\tau_{\phi,F}(n)\leq(n^m+1)^t$ for every $n\in\N_+$ with $n\geq m$. Let $P_F(x)=(x^m+1)^t+k\in\Z[x]$, where $k=\tau_{\phi,F}(m)$. Then $\tau_{\f,F}(n)\leq P_F(n)$ for every $n\in\N_+$, and this shows that $\f\in\pgp_F$. Since $F$ was chosen arbitrary, we have $\f\in\pgp$. This concludes the proof of \eqref{Pm-1->Pm-gen}.

\smallskip
By \eqref{Pm-1->Pm-gen}, $\f\restriction_{P_n(G,\f)}\in\pgp$ for every $n\in\N$. On the other hand, $\f\restriction_{\ker_\infty \f}\in\pgp$ by  Lemma \ref{pgp-lemma}(d), hence
by Proposition \ref{properties}(c) and Lemma \ref{Q=P+K}, it follows that $\f\restriction_{Q_n(G,\f)}\in\pgp$ for every $n\in\N$. Since $\QQ(G,\f)$ is an increasing union of the subgroups $Q_n(G,\f)$, $\f\restriction_{\QQ(G,\f)}\in\pgp$ by Proposition \ref{properties}(d).
\end{proof}

\begin{corollary}\label{Q<pgp}
For every algebraic flow $(G,\phi)$, we have $\QQ(G,\phi)\subseteq\pgp(G,\phi)$.
\end{corollary}

\begin{corollary}
Every locally quasi-periodic endomorphism has polynomial growth.
\end{corollary}


\section{Two characterizations of the Pinsker subgroup}\label{Q-P^h-sec}

\begin{lemma}\label{pgp->h=0}
Let $(G,\phi)$ be an algebraic flow.
\begin{itemize}
\item[(a)] If $\f\in\pgp_F$ for some $F\in[G]^{<\omega}$, then $H(\phi,F)=0$.
\item[(b)] If $\f\in\pgp$, then $h(\f)=0$.
\end{itemize}
\end{lemma}
\begin{proof}
(a) By definition there exists $P_F(x)\in\Z[x]$ such that $\tau_{\phi,F}(n)\leq P_F(n)$ for every $n\in\N_+$. Then
 $$
 H(\phi,F)=\lim_{n\to\infty}\frac{\log\tau_{\phi,F}(n)}{n}\leq\lim_{n\to\infty}\frac{\log P_F(n)}{n}=0. 
 $$
(b) Follows from (a).
\end{proof}

\begin{corollary}\label{Q<PGP<P}
Let $(G,\phi)$ be an algebraic flow. Then $\QQ(G,\f)\subseteq\pgp(G,\f)\subseteq\P(G,\f)$.
\end{corollary}
\begin{proof}
The inclusion $\QQ(G,\f)\subseteq\pgp(G,\f)$ is proved in Corollary \ref{Q<pgp}. By Lemma \ref{pgp->h=0}(b), $\pgp(G,\f)\subseteq\P(G,\f)$.
\end{proof}

\begin{remark}\label{uff-rem}
Let $(G,\phi)$ be an algebraic flow. It is possible to prove directly that $\QQ(G,\f)\subseteq\P(G,\f)$, using
the inclusion $Q_1(G,\f)\subseteq\P(G,\f)$ from Proposition \ref{t=Q1} and the following weaker form of the Addition Theorem \ref{AT}: if $H$ is a $\phi$-invariant subgroup of $G$, then $h(\phi)=0$ provided that $h(\phi\restriction_H)=0=h(\overline\phi)$, where $\overline\phi:G/H\to G/H$ is the endomorphism induced by $\phi$. 
\end{remark}

The inclusion $\QQ(G,\f)\subseteq\P(G,\f)$ from Corollary \ref{Q<PGP<P} gives 

\begin{corollary}\label{LastCoro} 
Let $(G,\phi)$ be an algebraic flow. If $\phi$ has completely positive algebraic entropy then $\phi$ is algebraically ergodic. 
\end{corollary}

We shall see below that this implication can be inverted (see Corollary \ref{blablabla}). 

\begin{corollary}
Let $(G,\phi)$ be an algebraic flow. If $G$ is torsion, then $\P(G,\f)=t_\phi(G)=\pgp(G,\phi)=\QQ(G,\f)=Q_1(G,\f)$.
\end{corollary}
\begin{proof}
By Proposition \ref{t=Q1}, $\P(G,\f)=t_\phi(G)=Q_1(G,\f)\subseteq \QQ(G,\f)$, so Corollary \ref{Q<PGP<P} applies.
\end{proof}

The following theorem due to Kronecker will be needed in the next proof: 

\begin{theorem}[Kronecker Theorem]\label{Kronecker}\emph{\cite{Kr}}
Let $\alpha$ be a non-zero algebraic integer and let $f(t)\in\Z[t]$ be its minimal polynomial over $\Q$. If all the roots of $f(t)$ have absolute value $\leq 1$, then $\alpha$ is a root of unity.
\end{theorem}

In the next proposition we show that for a non-trivial abelian group an endomorphism of zero algebraic entropy is not algebraically ergodic.

\begin{proposition}\label{h=0->qp}
Let $(G,\phi)$ be an algebraic flow. If $G\neq0$ and $h(\phi)=0$, then $Q_1(G,\phi)\neq 0$.
\end{proposition}
\begin{proof}
We split the proof in several steps, restricting the problem to the case of an automorphism of $\Q^n$ for some $n\in\N_+$.

\smallskip
(a) We can suppose that $\phi$ is injective. Indeed, if $\phi$ is not injective, then there are certainly non-zero quasi-periodic elements, as $\ker\f\neq0$. 

\smallskip
(b) We can suppose that $G$ is torsion-free. In fact, if $G$ is torsion, $G=\P(G,\f)=Q_1(G,\f)$ by Proposition \ref{t=Q1}, and so $Q_1(G,\f)\neq0$. If $G$ has non-trivial torsion elements, then $t(G)\ne0$, and so $Q_1(t(G),\phi\restriction_{t(G)}) \neq 0$ by the torsion case. 

\smallskip
(c) We can suppose that $G$ is a divisible torsion-free abelian group. Indeed, by (b) we can assume that $G$ is torsion-free. Let $D$ be the divisible hull of $G$ and $\widetilde\f:D\to D$ the (unique) extension of $\f$ to $D$. By Lemma \ref{AA}, $h(\widetilde\f)=0$.
Assume that $Q_1(D,\widetilde\f)\neq 0$. Since $G$ is essential in $D$ and $Q_1(G,\f)=Q_1(D,\widetilde\f)\cap G$ by Lemma \ref{Q_H}(a), it follows that also $Q_1(G,\f)\neq 0$.

\smallskip
(d) We can suppose that $G$ is a divisible torsion-free abelian group of finite rank. Indeed, if there exists a non-zero element $x\in G$ with $V(\phi,x)$ of infinite rank, then $h(\f\restriction_{V(\phi,x)})=\infty$, since $\phi\restriction_{V(\phi,x)}$ is conjugated to the right Bernoulli shift $\beta_\Z$ with $h(\beta_\Z)=\infty$ (see Example \ref{beta}) and so Fact \ref{conjugation_by_iso}(b) applies. By Fact \ref{conjugation_by_iso}(a) this implies $h(\phi)=\infty$, against our hypothesis. Then $V(\phi,x)$ has finite rank for every $x\in G$. Moreover, each $V(\phi,x)$ is $\phi$-invariant. So we can assume without loss of generality that $G$ has finite rank and  $G=V(\phi,x)$ for some $x\in G$.

\smallskip
(e) Suppose that $G$ is a divisible torsion-free abelian group of finite rank $n> 0$. Then $G\cong \Q^n$ and by (a) we can assume that $\f$ is an automorphism of $G$. Let $A$ be the $n\times n$ matrix over $\Q$ of $\phi$, and let $P(t)$ be the characteristic polynomial of $A$.
Since $h(\phi)=0$, by (\ref{Yuz}) of Theorem \ref{Yuz} $s=1$ and all the eigenvalues $\lambda_i$ of $A$ have $|\lambda_i|=1$. In other words, $P(t)$ is a monic polynomial with all roots of modulus $1$. By Theorem \ref{Kronecker}, all the roots of $P(t)$ are roots of the unity. Then there exist a non-zero $x\in G$ and $m\in\N_+$ such that $\phi^m(x)=x$, that is, $x$ is a non-zero periodic point of $\phi$. In particular, $Q_1(G,\phi)\neq 0$.
\end{proof}

\begin{remark}
Note that in (e) we apply the Algebraic Yuzvinski Formula (\ref{Yuz}). This is its unique application in this paper, but Proposition \ref{h=0->qp} is a fundamental step in proving our main result, that is, Theorem \ref{Q=PGP=P}. This leaves open the question of whether it is possible to prove Proposition \ref{h=0->qp} without using the Algebraic Yuzvinski Formula. It would be sufficient to carry out the last step (e), i.e., prove that if $h(\phi)= 0$ for an automorphism $\phi$ of $\Q^n$, then $\phi$ has non-zero periodic points. 
\end{remark}

\begin{corollary}\label{blablabla}
Let $(G,\phi)$ be an algebraic flow. Then $\phi$ is algebraically ergodic if and only if $\phi$ has completely positive algebraic entropy. 
\end{corollary}
\begin{proof} According to Corollary \ref{LastCoro} it sufficies to prove (equivalently) that
\begin{equation}\label{0->0}
\QQ(G,\f)=0\ \Longrightarrow\ \P(G,\f)=0.
\end{equation}
To this end, let $H=\P(G,\f)$.  Then $\QQ(H,\f\restriction_H)=\QQ(G,\f)\cap H$ by Lemma \ref{Q_H}(b). By hypothesis $\QQ(G,\f)=0$, and so $\QQ(H,\f\restriction_H)=0$ as well. In particular, $Q_1(H,\f\restriction_H)=0$. So the assumption $\P(G,\f)\ne 0$, along with 
Proposition \ref{h=0->qp},  would give $h(\f\restriction_H)>0$, a contradiction. Hence, $\P(G,\f)=0$, and \eqref{0->0} is proved.
\end{proof}

Now we are in position to prove our Main Theorem: 

\begin{theorem}\label{Q=PGP=P}
Let $(G,\phi)$ be an algebraic flow. Then $\QQ(G,\f)=\pgp(G,\f)=\P(G,\f)$.
\end{theorem}
\begin{proof} By Corollary \ref{Q<PGP<P}, $\QQ(G,\f)\subseteq \pgp(G,\f) \subseteq \P(G,\f)$.   To prove that $\P(G,\f)\subseteq \QQ(G,\f)$, let $\overline\f:G/\QQ(G,\f)\to G/\QQ(G,\f)$ be the endomorphism induced by $\f$. By Proposition \ref{P.T}(a) $\QQ(G/\QQ(G,\f),\overline\f)=0$ and so 
\begin{equation}\label{LAAAAAAAST}
\P(G/\QQ(G,\f),\overline\f)=0.
\end{equation}
 by \eqref{0->0}. Let $\pi:G\to G/\QQ(G,\f)$ be the canonical projection. Since $\pi(\P(G,\f))\subseteq\P(G/\QQ(G,\f),\overline\f)$ by Lemma \ref{pi<p}(b),   from (\ref{LAAAAAAAST}) we conclude that $\P(G,\f)\subseteq\ker\pi=\QQ(G,\f)$.
\end{proof}

The following is a direct consequence of Proposition \ref{P.T}(a) and Theorem \ref{Q=PGP=P}. It is possible to prove it using the Addition Theorem \ref{AT} without applying Theorem \ref{Q=PGP=P} (see Remark \ref{uff-rem}). 

\begin{corollary}
Let $(G,\phi)$ be an algebraic flow. Then the induced endomorphism $\overline \f: G /\P(G,\f)\to G /\P(G,\f)$ has $h(\overline\phi)>\!\!>0$, i.e., $\P(G/\P(G,\f), \overline \f)=0$.
 \end{corollary}

\begin{corollary}\label{0<->pgp}
Let $(G,\phi)$ be an algebraic flow. Then $h(\phi)=0$ if and only if $\f\in\pgp$. Consequently,  $h(\phi)>0$ if 
and only if there exists $F\in[G]^{<\omega}$ such that $\f\not\in\pgp_F$.
\end{corollary}
\begin{proof}
Since $h(\phi)=0$ if and only if $G=\P(G,\phi)$, and $\P(G,\phi)=\pgp(G,\phi)$ by Theorem \ref{Q=PGP=P}, we can conclude that $h(\phi)=0$ precisely when $\phi\in\pgp$.
\end{proof}

Corollary \ref{0<->pgp}  can be stated also as
\begin{equation}\label{forevery}
H(\phi,F)=0\ \text{for every}\ F\in[G]^{<\omega}\ \text{if and only if}\ \f\in\pgp_F\ \text{for every}\ F\in[G]^{<\omega}, 
\end{equation}
since $h(\phi)=0$ is equivalent to $H(\phi,F)=0$ for every $F\in[G]^{<\omega}$ and $\f\in\pgp$ is equivalent to $\f\in\pgp_F$ for every $F\in[G]^{<\omega}$. It is natural to ask if it possible to strengthen  \eqref{forevery} by removing the universal quantifier.  
Now we prove this more precise result, 
%
providing an important dichotomy for the growth of the algebraic entropy with respect to a non-empty finite subset. 

\begin{theorem}\label{exp}
Let $(G,\phi)$ be an algebraic flow and $F\in[G]^{<\omega}$. Then $$H(\phi,F)\begin{cases}>0 & \text{if and only if}\ \f\in\mathrm{Exp}_F,\\ =0 & \text{if and only if}\ \f\in\pgp_F.\end{cases}$$
\end{theorem}
\begin{proof}
We prove first that $H(\phi,F)>0$ if and only if $\f\in\mathrm{Exp}_F$.
Note that both $H(\phi,F)>0$ and $\f\in\mathrm{Exp}_F$ imply $|F|\geq2$; indeed, if $|F|=1$, then $\tau_{\phi,F}(n)=1$ for every $n\in\N_+$.

Assume that $H(\phi,F)=a>0$. Consequently, there exists $m\in\N_+$ such that $\log\tau_{\phi,F}(n)>n\cdot \frac{a}{2}$ for every $n>m$. Then $\tau_{\phi,F}(n)>e^{n\cdot \frac{a}{2}}$ for every $n>m$. Since $|F|\geq2$, $\tau_{\phi,F}(n)\geq2$ for every $n\in\N_+$; in particular, $\tau_{\phi,F}(n)\geq(\sqrt[m]{2})^n$ for every $n\leq m$. For $b=\min\{\sqrt[m]{2},e^{\frac{a}{2}}\}$, we have $\tau_{\phi,F}(n)\geq b^n$ for every $n\in\N_+$, and this proves that $\phi\in\mathrm{Exp}_F$.

Suppose now that $\phi\in\mathrm{Exp}_F$. Then there exists $b\in\R_+$, $b>1$, such that $\tau_{\phi,F}(n)\geq b^n$ for every $n\in\N_+$. Hence $H(\phi,F)\geq \log b>0$.

By Lemma \ref{pgp->h=0}(a) if $\f\in\pgp_F$, then $H(\f,F)=0$. 
If $H(\f,F)=0$,  by Lemma \ref{proexp} $h(\phi\restriction_{V(\phi,F)})=0$. Corollary \ref{0<->pgp} implies in particular that $\f\in\pgp_F$. 
\end{proof}

Note that this more precise form of \eqref{forevery} follows from \eqref{forevery} (since in the proof of Theorem \ref{exp} we apply Corollary \ref{0<->pgp}); so they are equivalent.


\section{Appendix: Connection with the topological entropy}\label{top-sec}

For an abelian group $G$ the Pontryagin dual $\widehat G$ is $\mathrm{Hom}(G,\T)$ endowed with the compact-open topology \cite{P}. The Pontryagin dual of an abelian group is compact. Moreover, for an endomorphism $\phi:G\to G$, its adjoint endomorphism $\widehat\phi:\widehat G\to\widehat G$ is continuous. For basic properties concerning the Pontryagin duality see \cite{DPS} and \cite{HR}. 
For a subset $A$ of $G$, the annihilator of $A$ in $\widehat G$ is $A^\perp=\{\chi\in\widehat G:\chi(A)=0\}$, while for a subset $B$ of $\widehat G$, the annihilator of $B$ in $G$ is $B^\perp=\{x\in G:\chi(x)=0\ \text{for every }\chi\in B\}$.

\smallskip
We recall the definition of the topological entropy following \cite{AKM}. For a compact topological space $X$ and for an open cover $\mathcal U$ of $X$, let $N(\mathcal U)$ be the minimal cardinality of a subcover of $\mathcal U$. Since $X$ is compact, $N(\mathcal U)$ is always finite. Let $H(\mathcal U)=\log N(\mathcal U)$ be the \emph{entropy of $\mathcal U$}.
For any two open covers $\mathcal U$ and $\mathcal V$ of $X$, let $\mathcal U\vee\mathcal V=\{U\cap V: U\in\mathcal U, V\in\mathcal V\}$. Define analogously $\mathcal U_1\vee\ldots	\vee\mathcal U_n$, for open covers  $\mathcal U_1,\ldots,\mathcal U_n$ of $X$.
Let $\psi:X\to X$ be a continuous map and $\mathcal U$ an open cover of $X$. Then $\psi^{-1}(\mathcal U)=\{\psi^{-1}(U):U\in\mathcal U\}$.
The \emph{topological entropy of $\psi$ with respect to $\mathcal U$} is $H_{top}(\psi,\mathcal U)=\lim_{n\to\infty}\frac{H(\mathcal U\vee\psi^{-1}(\mathcal U)\vee\ldots\vee\psi^{-n+1}(\mathcal U))}{n},$ and the \emph{topological entropy} of $\psi$ is $h_{top}(\psi)=\sup\{H_{top}(\psi,\mathcal U):\mathcal U\ \text{open cover of $X$}\}.$

\bigskip
The following theorem connects the topological entropy and the algebraic entropy making use of the Pontryagin duality.

\begin{theorem}\label{Peters}\emph{\cite{Pet,DG}}
Let $(G,\phi)$ be an algebraic flow. Then $h(\phi)=h_{top}(\widehat \phi)$.
\end{theorem}

Weiss \cite{W} proved this theorem for torsion abelian groups $G$, Peters \cite{Pet} extended it to automorphisms of arbitrary countable abelian groups $G$. 
Finally, the above general form for arbitrary algebraic flows  was deduced in \cite{DG} from these two cases, the Addition Theorem \ref{AT} and the facts that
(i) $G$ is countable if and only if $\widehat G$ is metrizable; (ii) the topological entropy is ``continuous" with respect to inverse limits.
  
\medskip
\noindent{\bf Proof of Theorem \ref{Yuz}.}
Let $\phi$ is an automorphism of $\Q^n$. Then $\phi$ is $\Q$-linear, so described by some $A\in GL_n(\Q)$. Let $\widehat \phi:\widehat \Q^n \to\widehat\Q^n$ be the adjoint automorphism of $\phi$, where $\widehat \Q$ is the Pontryagin dual of $\Q$; in particular, $\widehat\Q$ is a $\Q$-vector space. Then $\widehat\phi$ is described by the matrix $A^t\in GL_n(\Q)$, which is the transposed of $A$, i.e., $\widehat\phi(x)= A^tx^t$ for $x=(x_1, \ldots, x_n) \in \widehat \Q^n$. 
According to the Yuzvinski Formula for the topological entropy of automorphisms of $\widehat\Q^n$ (see \cite{LW,WardLN,Y}) 
$$
h_{top}(\widehat\f)=\log s+ \sum_{|\lambda_i|>1}\log |\lambda_i|,
$$
where $\lambda_i$ are the eigenvalues of $A^t$ and $s$ is the least common multiple of the denominators of the coefficients of the (monic) characteristic polynomial of $A^t$. Since the eigenvalues and the characteristic polynomials of $A$ and $A^t$ coincide, and since $h(\phi)=h_{top}(\widehat\phi)$ by Theorem \ref{Peters}, we are done. \hfill $\square$
\medskip

We start with a general comparison between the algebraic and the topological entropy.

\begin{remark}\label{pontr-rem}
\begin{itemize}
\item[(a)] Let $(G,\phi)$ be an algebraic flow. Let $K=\widehat G$ and $\psi=\widehat\phi$. Let also $H$ be a $\phi$-invariant subgroup of $G$. By the Pontryagin duality $N=H^\perp$ is a closed $\psi$-invariant subgroup of $K$, and $N^\perp=H$. Moreover, we have the following commutative diagrams:
\begin{equation*}
\xymatrix{
H\ar[d]_{\phi\restriction_H}\ar@{^{(}->}[r] & G \ar[d]^\phi\ar@{->>}[r] & G/H\ar[d]^{\overline\phi} & & K/N & K\ar@{->>}[l] & N \ar@{_{(}->}[l] \\
H\ar@{^{(}->}[r] & G\ar@{->>}[r] & G/H & & K/N\ar[u]^{\overline\psi} & K \ar[u]^\psi\ar@{->>}[l] & N\ar[u]_{\psi\restriction_N} \ar@{_{(}->}[l]
}
\end{equation*}
The second diagram is obtained by the first one applying the Pontryagin duality functor. In particular, $\widehat{K/N}\cong H$ and $\widehat N\cong G/H$. Moreover, $\widehat{\overline\psi}$ is conjugated to $\phi\restriction_H$ and $\widehat{\psi\restriction_N}$ is conjugated to $\overline \phi$.
\item[(b)] Using Theorem \ref{Peters}, Fact \ref{conjugation_by_iso}(b) and item (a), we have:
$$h(\phi\restriction_H)=h_{top}(\overline\psi)\ \ \text{and}\ \ h(\overline\phi)=h_{top}(\psi\restriction_N).$$
\end{itemize}
\end{remark}

The following theorem connects the algebraic Pinsker subgroup with the topological Pinsker factor. In the sequel, for a compact abelian group and a continuous automorphism 
$\psi:K\to K$ of $K$ we denote by $\mathcal E(K,\psi)$  the (closed) subgroup $\P(\widehat K,\widehat\phi)^\perp$ of $K$: 

\begin{theorem}\label{P-P}
Let $K$ be a compact abelian group and $\psi:K\to K$ a continuous endomorphism of $K$. Then 
the  induced endomorphism   $\overline\psi:K/\mathcal E(K,\psi)\to K/\mathcal E(K,\psi)$
is the topological Pinsker factor of $(K,\psi)$.
\end{theorem}
\begin{proof}
Let $G=\widehat K$ and $\phi=\widehat \psi$. 
Apply Remark \ref{pontr-rem}(b) for $H=\P(G,\phi)$; then $h(\phi\restriction_{\P(G,\phi)})=h_{top}(\overline\psi)$.
Since $\P(G,\phi)$ is the greatest $\phi$-invariant subgroup of $G$ where the restriction of $\phi$ has zero algebraic entropy, it follows that $(K/\mathcal E(K,\psi),\overline\psi)$ is the greatest factor of $(K,\psi)$ with zero topological entropy.
\end{proof}

The ergodic transformations are the core of ergodic theory (see \cite{Wa} for the definition and main properties). For our purposes it will be enough 
to recall here the following characterization of ergodicity of a continuous automorphism $\psi$ of a compact abelian group $K$ considered with its Haar measure proved independently by Halmos and Rohlin: $\psi$ is ergodic if and only if $\widehat\psi^n(x)= x$ for $x\in K$ and $n\in\N_+$ implies $x=0$. 
In other words: 

\begin{theorem}\label{ergodic-P1} {\rm \cite{Wa}}
A continuous automorphism $\psi:K\to K$ of a compact abelian group $K$ is ergodic if and only if $P_1(\widehat K,\widehat\psi)=0$.
\end{theorem}

The next fact justifies the introduction of the concept of algebraically ergodic endomorphism given in Definition \ref{alg-erg}.

\begin{proposition}\label{ergodic}
Let $K$ be a compact abelian group and $\psi:K\to K$ a continuous automorphism of $K$. Then $\psi$ is ergodic if and only if $\widehat\psi$ is algebraically ergodic.
\end{proposition}
\begin{proof}
Let $G=\widehat K$ and $\phi=\widehat\psi$. By Theorem \ref{ergodic-P1} $\psi$ is ergodic if and only if $P_1(G,\phi)=0$. Since $\phi$ is an automorphism, $Q_1(G,\phi)=P_1(G,\phi)=0$. By Lemma \ref{Q=0<->Q1=0} this is equivalent to $\QQ(G,\phi)=0$.
\end{proof}

The following corollary connects the ergodic theory, the topological entropy and the algebraic entropy.

\begin{corollary}\label{Last}
Let $K$ be a compact abelian group and $\psi:K\to K$ a continuous automorphism of $K$. Then the following conditions are equivalent:
\begin{itemize}
\item[(a)] $\psi$ is ergodic;
\item[(b)] $\psi$ has completely positive topological entropy;
\item[(c)] $\widehat\psi$ is algebraically ergodic;
\item[(d)] $\widehat\psi$ has completely positive algebraic entropy.
\end{itemize}
\end{corollary}

From the above results we obtain: 

\begin{corollary}
Let $K$ be a compact abelian group and $\psi:K\to K$ a continuous injective endomorphism of $K$.
Then the closed $\psi$-invariant subgroup $\mathcal E(K,\psi)$ of $K$ is the \emph{greatest domain of ergodicity} of $\psi$ in the following sense: 
\begin{itemize}
\item[(a)]  the restriction $\psi\restriction_{\mathcal E(K,\psi)}$ is ergodic; and 
\item[(b)]  $\mathcal E(K,\psi)$ is the greatest closed $\psi$-invariant subgroup of $K$ with the property (a).
\end{itemize}
\end{corollary}

\begin{proof}
Let $ G = \widehat{K}$ and $ \f=\widehat{\psi}: G \to G$. 
Our hypothesis implies that the adjoint endomorphism $\f $ is surjective. Then the induced endomorphism $\overline{\phi}: G/P_1(G,\phi)\to
 G/P_1(G,\phi)$ is an automorphism by Theorem \ref{Q=PGP=P} and Corollary \ref{NewCoro}. By Pontryagin duality, $G/P_1(G,\phi) $ coincides
 with the dual of $\mathcal E(K,\psi)$ and the dual of the automorphism $\overline{\phi}$
 coincides with the  restriction  $\psi\restriction_{\mathcal E(K,\psi)}: \mathcal E(K,\psi) \to \mathcal E(K,\psi)$. 
 Since $\overline{\phi}$ is algebraically ergodic by Proposition \ref{P.T}, one concludes with Corollary \ref{Last}  that $\psi\restriction_{\mathcal E(K,\psi)}$ is 
 ergodic. This proves (a). A similar argument, using item (b) of Proposition \ref{P.T} proves item (b). 
\end{proof}

\end{document}